\documentclass[a4paper,10pt,reqno]{amsart}

\usepackage[ansinew]{inputenc} 
\usepackage[T1]{fontenc}
\usepackage{textcomp}
\usepackage{mathcomp}
\usepackage{amsmath,amsthm,amssymb,amsfonts,amsxtra}
\usepackage{mathtools}
\usepackage{paralist}
\usepackage{float}
\usepackage{comment}
\usepackage{enumerate}
\usepackage{tikz-cd}

\newtheorem{thm}{Theorem}[section]
\newtheorem{cor}[thm]{Corollary}
\newtheorem{prop}[thm]{Proposition}
\newtheorem{lem}[thm]{Lemma}

\theoremstyle{definition}
\newtheorem{defn}[thm]{Definition}
\newtheorem{defns}[thm]{Definitions}
\newtheorem{notns}[thm]{Notations}

\theoremstyle{remark}
\newtheorem{rem}[thm]{Remark}
\newtheorem{rems}[thm]{Remarks}

\numberwithin{equation}{section}

\newcommand{\set}[2]{\left \{ \: #1 \: \middle | \: #2 \: \right \} }
\newcommand{\exx}[1]{e^{i2\pi #1 }}
\newcommand{\xp}[1]{\exp \left( #1 \right)}
\newcommand{\wh}[1]{\widehat{#1}}
\newcommand{\wt}[1]{\widetilde{#1}}
\newcommand{\ov}[1]{\overline{#1}}

\newcommand{\conv}[1]{\stackrel{#1}{\longrightarrow}}
\newcommand{\id}{\mathrm{id}}
\newcommand{\lspan}{\mathrm{span} \,}

\newcommand{\N}{\mathbb{N}}
\newcommand{\Z}{\mathbb{Z}}
\newcommand{\R}{\mathbb{R}}
\newcommand{\Q}{\mathbb{Q}}
\newcommand{\C}{\mathbb{C}}
\newcommand{\T}{\mathbb{T}}

\newcommand{\cC}{\mathcal{C}}
\newcommand{\cE}{\mathcal{E}}
\newcommand{\cG}{\mathcal{G}}
\newcommand{\cH}{\mathcal{H}}
\newcommand{\cK}{\mathcal{K}}
\newcommand{\cL}{\mathcal{L}}
\newcommand{\cN}{\mathcal{N}}
\newcommand{\cO}{\mathcal{O}}
\newcommand{\cP}{\mathcal{P}}
\newcommand{\cR}{\mathcal{R}}
\newcommand{\cW}{\mathcal{W}}

\title{Enveloping semigroups and quasi-discrete spectrum}

\author{Juho Rautio}

\address{University of Oulu\\
   Department of Mathematical Sciences\\
   PL 8000\\
   FI-90014 Oulun yliopisto\\
   Finland}

\email{juho.rautio@oulu.fi}

\thanks{This paper is part of the author's Phd.D. work at the University of Oulu. The author is grateful for the assistance of his adviser Mahmoud Filali.}



\begin{document}

\begin{abstract}
The structures of the enveloping semigroups of certain elementary finite- and infinite-dimensional distal dynamical systems are given, answering open problems posed by Namioka in 1982. The universal minimal system with (topological) quasi-discrete spectrum is obtained from the infinite-dimensional case. It is proved that, on one hand, a minimal system is a factor of this universal system if and only if its enveloping semigroup has quasi-discrete spectrum and that, on the other hand, such a factor need not have quasi-discrete spectrum in itself. This leads to a natural generalisation of the property of having quasi-discrete spectrum, which is named the $\cW$-property.
\end{abstract}

\maketitle

\section{Introduction}

The starting point of this work is the problem of describing the enveloping semigroups (or Ellis groups) of a particular class of distal dynamical systems. Although enveloping semigroups play an important part in the abstract theory of topological dynamics, concrete examples are relatively scarce, as Glasner points out in his survey \cite{glasner07}. The most prominent feature of the systems we are interested in is that they have quasi-discrete spectrum. This notion was first defined by Abramov in \cite{abramov62} in the setting of measure-preserving transformations, and the topological counterpart was introduced by Hahn and Parry in \cite{hahnparry65}. (We use the term `quasi-discrete spectrum' in the topological sense unless specified otherwise.) A simple example of a system with quasi-discrete spectrum is the skew product on the $2$-torus $\T^2$, $\T$ being the unit circle in $\C$, defined by the mapping
	\[(x,y) \mapsto (ax,xy) : \T^2 \to \T^2
\]
for some fixed constant $a \in \T$. The system is minimal if and only if $a = \exx{\alpha}$ for some irrational $\alpha \in \R$. The problem of finding the enveloping semigroup of this system, at least in the minimal case, goes back to articles of Furstenberg (\cite{furstenberg63}) and Namioka (\cite{namioka82}), the latter correcting the description given in the former. We can define analogous systems on $\T^m$ for any $m \geq 3$ or even $m = \N$, and finding their enveloping semigroups was left as an open problem in \cite{namioka82}. Some attempts and advances were made by Milnes in \cite{milnes86} and by Jabbari and Namioka in \cite{jabnam10}. Ultimately, the problem was left unsolved, but the topological centres were found in the latter work. The finite-dimensional case has been studied recently by Piku{\l}a (\cite{pikula10}), who showed that the Ellis groups are nilpotent on the basis of a partially explicit description of said groups. We present the structures of these groups in all the elementary cases in Theorems~\ref{sec:infdim} and \ref{sec:findim}. The Ellis groups of a class of distal systems containing those with quasi-discrete spectrum is given in quotient form in Theorem~\ref{sec:normalstructure} (see also Remarks~\ref{sec:wellis}). The key to all of these results is the space $\cH$ (defined in \ref{sec:H}), a closure of a certain family of self-homeomorphisms of $\T^{\Z_+}$, together with Lemma~\ref{sec:horbit}. In fact, we can endow $\cH$ with a group structure so that it becomes topologically isomorphic with the enveloping semigroup of any system on $\T^{\N}$ equipped with a mapping
	\[(x_1,x_2,x_3,\ldots) \mapsto (x_0 x_1, x_1 x_2, x_2 x_3,\ldots) : \T^{\N} \to \T^{\N},
\]
$x_0 \in \T$ being some constant. The techniques used in the proof of Lemma~\ref{sec:horbit} originate in \cite{namioka82}. Namely, we use Weyl's theorem on the distribution of polynomial sequences together with the Weyl criterion of uniform distribution in a compact abelian group.

We exploit the structure of $\cH$ to study systems with quasi-discrete spectrum, obtaining new results and clarifying and generalising some old ones. As Theorem~\ref{sec:hweyl}, we prove that $\cH$ can be seen as the spectrum of the so called Weyl algebra $\cW$, which is the closed subspace of $l^{\infty}(\Z)$ generated by all functions of the form $n \mapsto \exx{p(n)} : \Z \to \T$ where $p \colon \Z \to \R$ is a polynomial (see \cite{salehi91}), answering the second problem at the end of \cite{jabnam10}. These generators are connected to systems with quasi-discrete spectrum (Theorem~\ref{sec:qdsweyl}), and this connection was first studied, to the author's knowledge, by Hahn in \cite{hahn65}. In Corollary~\ref{sec:universal}, we deduce that $\cH$ is also the phase space of the universal minimal system with quasi-discrete spectrum, thus improving Brown's universal model in the topological case (\cite{brown69}). We can show that the system $\cH$ is uniquely ergodic, so all minimal systems with quasi-discrete spectrum are also uniquely ergodic (Theorem~\ref{sec:uew}).

In \cite{robertson74}, Robertson asked whether the factors of ergodic measure preserving transformations with quasi-discrete spectrum have quasi-discrete spectrum. This was proved to be true for totally ergodic transformations of standard probability spaces by Hahn and Parry in \cite{hahnparry68} some years earlier. One could also frame the question in the setting of topological dynamics: do the factors of minimal systems with quasi-discrete spectrum inherit these properties? In the last section, we shall construct an example showing that this is not the case. However, a detailed analysis of $\cH$ leads us to conclude that a minimal system is a factor of the universal minimal system with quasi-discrete spectrum if and only if the enveloping semigroup of said system has quasi-discrete spectrum; necessity is our Theorem~\ref{sec:wsystem}. This is achieved by finding all normal subgroups of the universal system that induce a Hausdorff quotient space. We introduce a generalisation of quasi-discrete spectrum, the $\cW$-property, on the basis of Theorem~\ref{sec:wsystem}. It is weaker than having quasi-discrete spectrum, and it is preserved in factors, subsystems, products and inverse limits.


\section{Preliminaries}

All topological spaces are assumed to be Hausdorff. We use $\N$ to denote the natural numbers beginning from $1$, and $\Z_+$ is the set of non-negative integers. The circle group $\T$ is understood as the space of complex numbers of modulus $1$. When convenient, we use the notation $\xp{\theta} = \exx{\theta}$, $\theta \in \R$. For a topological space $X$, the $C^{\ast}$-algebra of continuous, bounded, complex-valued functions on $X$ with the supremum norm is denoted by $\cC(X)$. If $X$ and $Y$ are topological spaces and $\pi \colon X \to Y$ is a continuous function, we define an adjoint mapping $\pi^{\ast} \colon \cC(Y) \to \cC(X)$ by $\pi^{\ast} f(x) = f(\pi(x))$ for all $f \in \cC(Y)$ and $x \in X$. We use the symbol $l^{\infty}$ to denote $l^{\infty}(\Z)$.

When $S$ is a semigroup, the \emph{left} and \emph{right translations} by $s \in S$ are the functions $\lambda_s, \rho_s \colon S \to S$, respectively, defined by $\lambda_s(t) = st$ and $\rho_s(t) = ts$, $t \in S$. We say that $S$ is \emph{right topological} if it carries a topology for which $\rho_s$ is continuous for each $s \in S$. The \emph{topological centre} of a right topological semigroup $S$ is the subsemigroup
	\[\Lambda(S) = \set{t \in S}{\lambda_t \text{ is continuous}}.
\]
These conventions follow \cite{bjm}.

We recall some of the key ideas of topological dynamics. For more information, see \cite{devries}. A \emph{dynamical system} or simply a \emph{system} is a pair $(X,s)$ where $X$ is a compact space, which we call the \emph{phase space} of the system, and $s \colon X \to X$ is a homeomorphism. The \emph{enveloping semigroup} of a system $(X,s)$ is the closure of the group $\{s^n\}_{n \in \Z}$ of iterates of $s$ under composition in the product space $X^X$, and it is denoted by $E(X,s)$. It is a compact right topological semigroup (under composition) with $\{s^n\}_{n \in \Z} \subseteq \Lambda(E(X,s))$. Note that, if $(X,s)$ is a dynamical system, so is $(E(X,s),\lambda_s)$.

Given a collection $\{(X_i,s_i)\}_{i \in I}$ of dynamical systems, their \emph{product system} is $(X,s)$ where $X = \prod_{i \in I} X_i$ and $s(x)_i = s_i(x_i)$, $x \in X$, $i \in I$.

If $(X,s)$ is a dynamical system, we say that a set $Y \subseteq X$ is \emph{invariant} if $s(Y) = Y$. If $Y \subseteq X$ is a non-empty, closed, invariant set, then the pair $(Y,s|_Y)$ is a \emph{subsystem} of $(X,s)$, and we tend to use the notation $(Y,s)$ for it. The \emph{orbit} of a point $x \in X$ is the invariant set $\cO(x) = \set{s^n x}{n \in \Z}$, and the \emph{positive semi-orbit} of $x$ is $\set{s^n x}{n \in \Z_+}$. The \emph{orbit closure} of $x \in X$ is the closure of its orbit in $X$. A \emph{transitive point} is a point $x \in X$ for which $\ov{\cO(x)} = X$, and if such a point exists, then the system $(X,s)$ is said to be \emph{point transitive}. Note that the identity mapping $\id_X = s^0$ is always a transitive point in $(E(X,s),\lambda_s)$, and the latter can be seen as a subsystem of products of $(X,s)$. We say that $(X,s)$ is \emph{minimal} if $X$ contains no proper, non-empty, closed, invariant sets. This property is equivalent to all points being transitive. It can also be shown that a system is minimal if and only if all positive semi-orbits are dense.

Consider two systems $(X,s)$ and $(Y,t)$. We say that a function $\pi \colon X \to Y$ is a \emph{homomorphism} if it is continuous, surjective and satisfies $\pi \circ s = t \circ \pi$. If such a mapping exists, we say that $(Y,t)$ is a \emph{factor} of $(X,s)$. An injective homomorphism is an \emph{isomorphism}. A homomorphism $\pi \colon X \to Y$ induces a continuous, surjective semigroup homomorphism $\wt{\pi} \colon E(X,s) \to E(Y,t)$ that satisfies (and is defined uniquely by) $\wt{\pi}(a)(x) = \pi(ax)$ for all $a \in E(X,s)$ and $x \in X$. We call it the \emph{induced homomorphism}. If $\pi$ is an isomorphism of flows, then $\wt{\pi}$ is an isomorphism of semigroups. The induced homomorphism is, of course, also a homomorphism from $(E(X,s),\lambda_s)$ to $(E(Y,t),\lambda_t)$. Note that the mapping $a \mapsto \lambda_a \colon E(X,s) \to E(E(X,s),\lambda_s)$ is an isomorphism of semigroups and of flows.

A system $(X,s)$ is \emph{distal} if the members of $E(X,s)$ are injective, i.e., taking any distinct points $x,y \in X$, the set of all pairs $(s^n(x),s^n(y))$, $n \in \Z$, has no cluster points on the diagonal of $X$. A well-known characterisation of distality states that a system $(X,s)$ is distal if and only if $E(X,s)$ is a group with $\id_X$ as its identity element, so the elements of $E(X,s)$ are bijections in this case. The enveloping semigroup of a distal system is often called its \emph{Ellis group}, and although this term has another meaning in the study of minimal flows (see 4.9, chapter 5, \cite{devries}), we shall use it from this point onwards. A second characterisation of distality is the following: a system $(X,s)$ is distal if and only if $(E(X,s),\lambda_s)$ is minimal or, equivalently, distal.

Suppose that $E$ is a compact, right topological group with an element $g \in \Lambda(E)$ such that $\{g^n\}_{n \in \Z}$ is dense in $E$. Let $T = \lambda_g$. Then, $(E,T)$ is a minimal distal system, and it is isomorphic to $(E(E,T),\lambda_T)$. Consequently, we can define an action of $E$ on any phase space of a factor of $(E,T)$ via the induced homomorphism: if $(X,s)$ is a dynamical system and $\pi \colon E \to X$ a homomorphism and $\wt{\pi} \colon E(E,T) \to E(X,s)$ the induced homomorphism, we define $ax = \wt{\pi}(\lambda_a)x$ for all $a \in E$ and $x \in X$. (The mapping $(a,x) \mapsto ax : E \times X \to X$ does not depend on the choice of $\pi$, provided that some homomorphism $\pi$ exists.) It is clear that $(ab)x = a(bx)$ for all $a,b \in E$ and $x \in X$. If $e \in E$ denotes the identity element, then for any $x \in X$, $ex = x$, and the mapping $a \mapsto ax : E \to X$ is continuous.

The factors of $(E,T)$ are in one-to-one correspondence with certain subgroups of $E$. To clarify this, suppose first that $(X,s)$ is a factor of $(E,T)$. If $\pi \colon E \to X$ is a homomorphism and if $x_0 = \pi(e)$, we can define $\cG = \cG(X,s,x_0) \subseteq E$ by
	\[\cG = \set{a \in E}{ax_0 = x_0}.
\]
It is easy to check that $\cG$ is a closed subgroup of $E$. Note also that $\cG = \pi^{-1}(x_0)$ since, for any $a \in E$, we have $ax_0 = \wt{\pi}(\lambda_a) \pi(e) = \pi(a)$. In addition, $\pi^{-1}(ax_0) = a\cG$ for any $a \in E$. Hence, the quotient space $E/\cG$ is canonically homeomorphic to $X$: if $a,b \in E$, then $a\cG = b\cG$ if and only if $\pi(a) = \pi(b)$. In particular, the quotient relation on $E$ induced by $\cG$ is closed. (This can be seen alternatively by evoking the $\tau$-topology on $E$; see \cite{auslander} or \cite{devries} for more information on this topic.) If another factor of $(E,T)$ produces the same group $\cG$ under this process, then it is easy to check that this factor is isomorphic to $(X,s)$. For the converse construction, suppose that we are given a subgroup $\cG^{\prime}$ of $E$ for which the quotient space $X^{\prime} = E/\cG^{\prime}$ is Hausdorff or, equivalently, the quotient relation is closed. Then, letting $\pi^{\prime} \colon E \to X^{\prime}$ be the quotient map and defining a function $s^{\prime} \colon X^{\prime} \to X^{\prime}$ by $s^{\prime}(a\cG^{\prime}) = ga\cG^{\prime}$, $a\cG^{\prime} \in X^{\prime}$, the pair $(X^{\prime},s^{\prime})$ is a dynamical system and $\pi^{\prime}$ is a homomorphism. Let $x_0^{\prime} = \pi^{\prime}(e)$. Now, the group $\cG(X^{\prime},s^{\prime},x_0^{\prime})$ coincides with $\cG^{\prime}$. Similarly, there is a one-to-one correspondence between the Ellis groups of factors of $(E,T)$ (which are again factors of $(E,T)$) and those normal subgroups $\cN$ of $E$ inducing a Hausdorff quotient space $E/\cN$, a compact right topological group.

Point transitive systems can be represented as certain kinds of subalgebras of $l^{\infty}$. If $(X,s)$ is a dynamical system and $x \in X$, we define a bounded linear operator $\omega_x \colon \cC(X) \to l^{\infty}$ by $\omega_x(f)(n) = f(s^n x)$ for all $f \in \cC(X)$ and $n \in \Z$. Let $R \colon l^{\infty} \to l^{\infty}$ be the translation operator: $R(f)(n) = f(n+1)$, $f \in l^{\infty}$, $n \in \Z$. Note that $\omega_x$ preserves complex conjugation, it maps constant functions to constant sequences, and $\omega_{x} \circ s^{\ast} = R \circ \omega_x$, so $\omega_x\cC(X)$ is a translation invariant $C^{\ast}$-subalgebra of $l^{\infty}$ containing the constants. If the point $x$ is transitive, then $\omega_x$ is an isometry. Conversely, given a translation invariant $C^{\ast}$-subalgebra $A \subseteq l^{\infty}$ that contains the constants, one can construct a related dynamical system $(X,s)$ with a transitive point $x \in X$ so that $\omega_x\cC(X) = A$, see section 5 of chapter 4 of \cite{devries}. Proposition 5.12 from this source, adapted to dynamical systems, states that, if $(X,s)$ and $(Y,t)$ are systems with transitive points $x \in X$ and $y \in Y$, respectively, then there exists a homomorphism $\pi \colon X \to Y$ with $\pi(x) = y$ if and only if $\omega_y\cC(Y) \subseteq \omega_x\cC(X)$. The homomorphism $\pi$ is an isomorphism if and only if $\omega_y\cC(Y) = \omega_x\cC(X)$. This theory can be seen as an extension of the theory of semigroup compactifications as outlined in \cite{bjm}; $m$-admissible subalgebras of $l^{\infty}$ correspond to enveloping semigroup systems.

Let $(X,s)$ be a dynamical system. By an \emph{invariant measure} on $X$ we mean a regular probability measure $\mu$ defined on the Borel $\sigma$-algebra $B(X)$ so that $\mu(s^{-1}(B)) = \mu(B)$ for every $B \in B(X)$. The system $(X,s)$ is \emph{uniquely ergodic} if there exists only one such $\mu$. The unique ergodicity of a system whose phase space contains a dense positive semi-orbit can be determined by examining the functions in the associated subalgebra of $l^{\infty}$:

\begin{thm}
\label{sec:ueambit}
Let $(X,s)$ be a dynamical system with a point $x \in X$ whose positive semi-orbit is dense in $X$. The system $(X,s)$ is uniquely ergodic if and only if there exists a family $F \subseteq \cC(X)$ such that $\lspan F$ is dense in $\cC(X)$ and, for every $f \in \omega_x F$, there is some $A(f) \in \C$ so that the average $(1/N) \sum_{n = 0}^{N-1} f(n+k)$ converges to $A(f)$ uniformly in $k \in \Z_+$ as $N$ grows.
\end{thm}

The proof is based on the well-known characterisation of unique ergodicity in terms of the convergence of the time averages (10.6.8 in \cite{devries}). This theorem implies that all factors of minimal uniquely ergodic systems are uniquely ergodic.

We quote a key theorem on distributions of sequences on $\T$ arising from polynomials.

\begin{thm}
\label{sec:poly}
Let $p \colon \N \to \R$ be a polynomial of degree $m \in \N$, say $p(n) = \sum_{k=0}^m \theta_k n^k$, $n \in \N$, $\theta_k \in \R$ for all $0 \leq k \leq m$. Then, there exists a constant $c \in \C$ such that
	\[\frac{1}{N} \sum_{n = 1}^N \exx{p(n+k)} \conv{N} c
\]
uniformly in $k \in \Z_+$. We have $c = 0$ in the following cases:
\begin{itemize}
\item[(i)] $\theta_k \in \R \setminus \Q$ for some $k \neq 0$;
\item[(ii)] $m = 1$ and $\theta_1 \notin \Z$.
\end{itemize}
\end{thm}

Case (i) says that the sequence $(p(n))$ in $\R$ is well distributed modulo $1$ if at least one of the coefficients $\theta_k$, $k \neq 0$, is irrational (Theorem 2 in \cite{lawton59}), and the sequence $(\exx{p(n)})$ is uniformly distributed in $\T$ (see \cite{kuipers} for definitions). If the polynomial $p(n) = \sum_{k=0}^m \theta_k n^k$ is such that $\theta_k \in \Q$ for every $k \geq 1$, then the sequence $(\exx{p(n)})$ is periodic, say $\exx{p(n + t)} = \exx{p(n)}$ for some $t \in \N$, and the term $(1/N) \sum_{n = 1}^N \exx{p(n+k)}$ converges to $(1/t) \sum_{n = 1}^t \exx{p(n)}$ uniformly in $k \in \Z_+$. This average is $0$ if $m = 1$ and $\theta_1 \notin \Z$.

Next, we give an overview of the necessary ideas and constructions relating to abelian topological groups. Most of the facts covered here can be found in \cite{hewittross} in one form or another.

If $G$ is an abelian group, we define a group homomorphism $n^{\times} \colon G \to G$ for each $n \in \Z$ by $n^{\times}(g) = g^n$, $g \in G$. We define $H(G)$ to be the set of group homomorphisms from $G$ to $\T$. It is a closed subspace of the product space $\T^G$, so it is compact, and we define a group operation on $H(G)$ by $(\phi \psi)(x) = \phi(x) \psi(x)$ for all $\phi,\psi \in H(G)$ and $x \in G$. In other words, $H(G)$ is the character group of the discretised $G$. As such, $H(G)$ is a compact, abelian topological group.

We are interested in $H(\T)$, the \emph{endomorphism group} of $\T$. Obviously, the character group $\wh{\T} = \{n^{\times}\}_{n \in \Z}$ is contained in $H(\T)$, and in fact the homomorphism $n \mapsto n^{\times} : \Z \to H(\T)$ realises the universal topological group compactification of $\Z$. In particular, the character group of $\T$ is dense in $H(\T)$, but there is an abundance of non-continuous functions in $H(\T)$. To get a glimpse of this, observe that, for any $x,y \in \T$ with $y = \exx{\theta}$ for some $\theta \in \R \setminus \Q$, there exists a function $\phi \in H(\T)$ such that $\phi(y) = x$. This follows from the fact that the rotation of the circle by angle $\theta$ defines a minimal system $(\T,\lambda_y)$ (Proposition 1.4 in chapter 3 of \cite{devries}), so the orbit of $1$ is dense, and there is a net $(n_\lambda)$ in $\Z$ for which $y^{n_\lambda} \conv{\lambda} x$. Any cluster point of the net $(n_\lambda^{\times})$ can serve as the desired $\phi$, and if $x$ is not in the cyclic group generated by $y$, the function $\phi$ is not continuous.

We define the following subgroups of $\T$:
\begin{align*}
\T_Q &= \set{\exx{q}}{q \in \Q}, \\
\T_k &= \set{\exx{(n/k)}}{n \in \Z} = \set{x \in \T}{x^k = 1}, k \in \N.
\end{align*}
The group $\T_Q$ is the union of the finite cyclic groups $\T_k$, and it is the torsion subgroup of $\T$. Hence, $\T$ is isomorphic to the direct product $\T/\T_Q \times \T_Q$ as a group. To construct an isomorphism between these two, pick a Hamel basis $B_1$ for $\R$ viewed as a vector space over $\Q$ so that $1 \in B_1$. Then, define $B = B_1 \setminus \{1\}$, and let $F_1 \colon \R \to \Q$ be the linear mapping that picks the scalar multiplier of $1$ in the representation of a vector as a linear combination of elements from $B_1$, i.e.,
	\[F_1 \left( q + \sum_{j=1}^{k} q_j b_j \right) = q
\]
for all $q,q_j \in \Q$, $b_j \in B$, $1 \leq j \leq k$, $k \in \N$. We can now define group homomorphisms $\Phi_1 \colon \T \to \T/\T_Q$, $\Phi_2 \colon \T \to \T_Q$ and $\Phi \colon \T \to \T/\T_Q \times \T_Q$ by
\begin{align*}
\Phi_1(x) &= x \T_Q,\\
\Phi_2(\exx{\theta}) &= \exx{F_1(\theta)},\\
\Phi(x) &= (\Phi_1(x),\Phi_2(x))
\end{align*}
for all $x \in \T$, $\theta \in \R$. The mapping $\Phi$ is an isomorphism. This structure allows us to find a sufficiently explicit description of $H(\T)$. To be precise, the latter is topologically isomorphic to the direct product $H(\T/\T_Q) \times H(\T_Q)$: we define a continuous isomorphism $\Psi \colon H(\T/\T_Q) \times H(\T_Q) \to H(\T)$ by
	\[\Psi(\phi,\psi) = (\phi \circ \Phi_1) (\psi \circ \Phi_2)
\]
for all $(\phi, \psi) \in H(\T/\T_Q) \times H(\T_Q)$. The inverse of $\Psi$ has coordinate functions $(\Psi^{-1})_1 \colon H(\T) \to H(\T/\T_Q)$ and $(\Psi^{-1})_2 \colon H(\T) \to H(\T_Q)$ so that
\begin{align*}
(\Psi^{-1})_1(\phi)(x \T_Q) &= \phi(x \Phi_2(\ov{x})), \\
(\Psi^{-1})_2(\phi)(y) &= \phi(y),
\end{align*}
$\phi \in H(\T), x \in \T, y \in \T_Q$.

We need to understand the structure of $H(\T/\T_Q)$. To this end, we define the $a$\emph{-adic solenoid} $\Sigma_a$ for $a = (2,3,4,\ldots)$ as
	\[\Sigma_a = \set{x \in \T^{\N}}{x_k^k = x_{k-1} \text{ for all } k \geq 2}.
\]
This is a closed subgroup of the product group $\T^{\N}$, so it is a compact, abelian topological group. The group $H(\T/\T_Q)$ is topologically isomorphic to the product group $\Sigma_a^B$. Indeed, we can define a continuous isomorphism $\Theta \colon H(\T/\T_Q) \to \Sigma_a^B$ by the formula
	\[\Theta(\phi)(b) = (\phi(\exx{(b/k!)}\T_Q))_{k = 1}^{\infty},
\]
$\phi \in H(\T/\T_Q)$, $b \in B$. The character group of $\Sigma_a$ is algebraically isomorphic to $\Q$. The mapping $J \colon \Q \to \wh{\Sigma_a}$, defined by $J(n/m!)(x) = x_m^n$ for all $n \in \Z$, $m \in \N$ and $x \in \Sigma_a$, is a well-defined group isomorphism. 

Note that, for all $k \in \N$ and $\phi \in H(\T)$, we have $\phi(\T_k) \subseteq \T_k$. In fact, there is some $n \in \Z$ depending on $\phi$ and $k$ so that $\phi(x) = x^n$ for all $x \in \T_k$. Also, $\phi(\T_Q) \subseteq \T_Q$. This means that $H(\T_Q)$ is a semigroup under composition. Moreover, the mapping
	\[(\phi,\psi) \mapsto \phi(\psi(x)) : H(\T) \times H(\T) \to \T
\]
is continuous for any $x \in \T_Q$, as the reader can verify, and composition is jointly continuous and commutative on $H(\T_Q)$. This is not the case for $H(\T)$, which is only a compact right topological semigroup under composition.

We conclude the preliminary section by recalling some basic facts concerning binomial coefficients. For each $n \in \Z$ and $k \in \N$, we define
\begin{align*}
\binom{n}{0} &= 1, \\
\binom{n}{k} &= \frac{n(n-1)(n-2)\cdots(n-k+1)}{k!}.
\end{align*}
This is a generalisation of the usual formula with $0 \leq k \leq n$, and the case of negative $n$ reduces to the non-negative case by the formula
	\[\binom{n}{k} = (-1)^k \binom{k-n-1}{k},
\]
$0 \leq k \leq n$. The numbers $\binom{n}{k}$ are integers, and the mapping $n \mapsto \binom{n}{k} \colon \Z \to \Z$ is a polynomial of degree $k$ for any $k \in \Z_+$. We may write
	\[\binom{n}{k} = \frac{1}{k!}\sum_{j=0}^{k}s(k,j)n^j
\]
where the integers $s(k,j)$ are called the \emph{Stirling numbers of the first kind}. Note that $s(0,0) = 1$, but $s(k,0) = 0$ for $k \geq 1$. Lastly, the generalised binomial coefficients satisfy some of the same classical identities as the usual ones: 
\begin{itemize}
\item[(i)] \emph{Pascal's rule:} $\binom{n}{k} + \binom{n}{k+1} = \binom{n+1}{k+1}$ for all $n \in \Z$ and $k \in \Z_+$;
\item[(ii)] \emph{Vandermonde's identity:} $\sum_{j=0}^{k} \binom{m}{j} \binom{n}{k-j} = \binom{m+n}{k}$ for all $m,n \in \Z$ and $k \in \Z_+$.
\end{itemize}


\section{Ellis groups}

This section is devoted to finding the Ellis groups of the dynamical systems described below.

\begin{defns}
Let $m \in \N$ and $x_0 \in \T$. By the $(m,x_0)$\emph{-system} we mean the dynamical system $(\T^m,r)$ where the homeomorphism $r \colon \T^m \to \T^m$ is defined by
	\[r(x_1,x_2,\ldots,x_m) = (x_0 x_1,x_1 x_2,\ldots,x_{m-1} x_m)
\]
for all $(x_k)_{k = 1}^m \in \T^m$. The $(\infty,x_0)$\emph{-system} is $(\T^{\N},s)$ where $s \colon \T^{\N} \to \T^{\N}$ is the homeomorphism
	\[s(x_1,x_2,x_3,\ldots) = (x_0 x_1,x_1 x_2,x_2 x_3,\ldots),
\]
$(x_k)_{k=1}^{\infty} \in \T^{\N}$. When we wish to refer to an $(m,x_0)$-system without specifying the parameters $m$ and $x_0$, we speak of a \emph{basic system}.
\end{defns}

\begin{notns}
If $G$ is a compact abelian group, we define a continuous group automorphism $A_G \colon G^{\Z_+} \to G^{\Z_+}$ by
	\[A_G(g_0,g_1,g_2,g_3,\ldots) = (g_0,g_0 g_1, g_1 g_2, g_2 g_3,\ldots),
\]
$(g_k)_{k=0}^{\infty} \in G^{\Z_+}$. We put $t = A_{\T}$ and $T = A_{H(\T)}$. We also define $\wt{0} \in H(\T)^{\Z_+}$ by
	\[\wt{0} = ( 1^{\times}, 0^{\times}, 0^{\times}, \ldots ) = (\id_{\T},1,1,\ldots),
\]
and we put $\wt{n} = T^n \wt{0}$ for all $n \in \Z$.
\end{notns}

Variants of our next proposition appear in many articles dealing with basic systems, e.g. \cite{hahnparry65} and \cite{pikula10}, but only for positive iterates, and explicit proofs tend to be omitted.

\begin{prop}
\label{sec:iter}
Let $G$ be a compact abelian group. For every $n \in \Z$, $k \in \Z_+$ and $g \in G^{\Z_+}$,
\begin{equation}
\label{iterate}
A_G^n(g)_k = \prod_{j=0}^{k} \binom{n}{k-j}^{\times} (g_j).
\end{equation}
In particular,
	\[\wt{n} = \left( \binom{n}{0}^{\times}, \binom{n}{1}^{\times}, \binom{n}{2}^{\times}, \ldots \right)
\]
for every $n \in \Z$.
\end{prop}

\begin{proof}
We use the notation $A = A_G$, and $g \in G^{\Z_+}$ is arbitrary. First, we prove the formula~\eqref{iterate} for $n \in \Z_+$ by induction. For $n=0$, we have $\binom{0}{k-j} = 0$ whenever $j < k$ and $\binom{0}{0} = 1$, so~\eqref{iterate} produces the identity map in this case. Suppose then that~\eqref{iterate} holds for some $n \in \Z_+$. Clearly,
	\[A^{n+1}(g)_0 = g_0 = \binom{n+1}{0}^{\times} (g_0),
\]
and for $k \geq 1$, using Pascal's rule (P.r.) we get
\begin{align*}
A^{n+1}(g)_k &= \prod_{j=0}^{k} \binom{n}{k-j}^{\times}(A(g)_j) = \binom{n}{k}^{\times}(g_0) \prod_{j=1}^{k} \binom{n}{k-j}^{\times}(g_{j-1} g_j) \\
&= \left[ \prod_{j=0}^{k-1} \left( \binom{n}{k-j-1} + \binom{n}{k-j} \right)^{\times} (g_j) \right] g_k \stackrel{P.r.}{=} \left[ \prod_{j=0}^{k-1} \binom{n+1}{k-j}^{\times} (g_j) \right] g_k \\
&= \prod_{j=0}^{k} \binom{n+1}{k-j}^{\times}(g_j).
\end{align*}
Thus, the non-negative iterates are of the form~\eqref{iterate}.

Next, we show that the inverse of $A$ is also given by~\eqref{iterate}. Now,
	\[\binom{-1}{0}^{\times}(A(g)_0) = A(g)_0 = g_0,
\]
and for $k \geq 1$,
\begin{align*}
\prod_{j=0}^k \binom{-1}{k-j}^{\times}(A(g)_j) &= \prod_{j=0}^k (-1)^{k-j \,\times}(A(g)_j) = (-1)^{k \, \times}(g_0) \prod_{j=1}^k (-1)^{k-j \,\times}(g_{j-1} g_j) \\
&= \left[ \prod_{j = 0}^{k-1} ((-1)^{k-j} + (-1)^{k-j-1})^{\times}g_j \right] g_k = g_k,
\end{align*}
as desired.

We repeat the inductive argument for negative values of $n$. Suppose that~\eqref{iterate} is true for some $n \leq -1$. Obviously,
	\[A^{n-1}(g)_0 = g_0 = \binom{n-1}{0}^{\times} (g_0),
\]
and for $k \geq 1$, with Vandermonde's identity (V.i.) we get
\begin{align*}
A^{n-1}(g)_k &= \prod_{l=0}^k \binom{n}{k-l}^{\times} (A^{-1}(g)_l) = \prod_{l=0}^k \prod_{j=0}^l \left( \binom{n}{k-l} \binom{-1}{l-j} \right)^{\times} (g_j) \\
&= \prod_{j=0}^k \prod_{l=j}^k \left( \binom{n}{k-l} \binom{-1}{l-j} \right)^{\times} (g_j) = \prod_{j=0}^k \prod_{l=0}^{k-j} \left( \binom{n}{k-j-l} \binom{-1}{l} \right)^{\times} (g_j) \\
&= \prod_{j=0}^k \left( \sum_{l=0}^{k-j} \binom{n}{k-j-l} \binom{-1}{l} \right)^{\times} (g_j) \stackrel{V.i.}{=} \prod_{j=0}^k \binom{n-1}{k-j}^{\times} (g_j),
\end{align*}
as required.

The statement concerning the elements $\wt{n}$ follows directly from \eqref{iterate}.
\end{proof}

The formula \eqref{iterate} can be used to show that $(G^{\Z_+},A_G)$ is distal for any compact abelian group $G$: simply pick distinct points $g,h \in G^{\Z_+}$, choose $k \in \Z_+$ as the smallest integer for which $g_k \neq h_k$, and compare the $k$th coordinates of $A_G^n(g)$ and $A_G^n(h)$ for any $n \in \Z$. As a corollary, the basic systems are distal as they can be obtained from $(\T^{\Z_+},t)$ by taking subsystems and factors, both of which preserve distality. Moreover, given $m \in \N \cup \{\infty\}$ and $x_0 \in \T$, the $(m,x_0)$-system is minimal if and only if $x_0 \in \T \setminus \T_Q$. The necessity is easy: if the $(m,x_0)$-system is minimal, so is its factor the $(1,x_0)$-system, and we know that latter is minimal if and only if $x_0 \in \T \setminus \T_Q$. In the finite-dimensional case, sufficiency can be proved by degree arguments, an approach used by Furstenberg in \cite{furstenberg61} and generalised by Parry in \cite{parry}, or by combining the Weyl criterion for the uniform distribution of sequences in compact abelian groups with Theorem~\ref{sec:poly} to conclude that all positive semi-orbits are uniformly distributed, as is done by Hahn in \cite{hahn65}. The infinite-dimensional case follows once we see the $(\infty,x_0)$-system with $x_0 \in \T \setminus \T_Q$ as an inverse limit of minimal systems (see Proposition 1.5 in chapter 4 of \cite{devries}). Alternatively, we can use the distribution arguments with virtually no modifications to achieve the same result.

Next, we define the spaces $\cH$ and $\cH_m$, $m \in \N$, which we shall use throughout the work to describe various Ellis groups and study systems with quasi-discrete spectrum.

\begin{notns}
\label{sec:H}
For any $\phi \in H(\T)$ and $n \in \N$, let $\phi^{(n)}$ be the composition of $n$ instances of $\phi$. For any $m \in \N$, let $\cH_m$ be the set of those sequences $(\phi_k)_{k = 0}^m \in H(\T)^{\{0,1,\ldots,m\}}$ that satisfy the following conditions:
\begin{align*}
\phi_0 &= 1^{\times}, \tag{H.0} \label{eq:H.0}\\
\phi_k(x^{k!}) &= \prod_{j=1}^k \phi_1^{(j)}(x^{s(k,j)}) \text{ for all } x \in \T_Q, 1 \leq k \leq m. \tag{H.\emph{m}} \label{eq:H.m}
\end{align*}
(Recall that $s(k,j)$ is a Stirling number of the first kind.) Let $\cH$ be the set of all sequences $(\phi_k)_{k = 0}^{\infty} \in H(\T)^{\Z_+}$ that satisfy \eqref{eq:H.0} and also the condition
\begin{equation*}
\phi_k(x^{k!}) = \prod_{j=1}^k \phi_1^{(j)}(x^{s(k,j)}) \text{ for all } x \in \T_Q, k \in \N. \tag{H.$\infty$} \label{eq:H.inf}
\end{equation*}
For any $m \in \N$, let $\rho_m \colon \cH \to \cH_m$ be the function that maps $(\phi_k)_{k = 0}^{\infty} \in \cH$ to $(\phi_k)_{k = 0}^m \in \cH_m$.
\end{notns}

The sets $\cH_m$ and $\cH$ are closed in $H(\T)^{\{0,1,\ldots,m\}}$ and $H(\T)^{\Z_+}$, respectively, for any $m \in \N$. Note also that $\cH_1 = \{1^{\times}\} \times H(\T)$, but for $m \geq 2$, the set $\cH_m$ is a proper subset of $\{1^{\times}\} \times H(\T)^m$. 

The following lemma takes care of the technicalities encountered in describing the Ellis groups of the basic systems.

\begin{lem}
\label{sec:horbit}
\mbox{}
\begin{itemize}
\item[(i)] The space $\cH$ is the orbit closure of $\wt{0}$ in $(H(\T)^{\Z_+},T)$ or, equivalently,
	\[\cH = \ov{\set{\wt{n}}{n \in \N}}.
\]
\item[(ii)] The function $\rho_m \colon \cH \to \cH_m$ is surjective for every $m \in \N$.
\item[(iii)] For any $m \in \N$, $(\phi_k)_{k = 0}^m \in \cH_m$ and $x_0 \in \T \setminus \T_Q$, we have
	\[\set{\phi_{m+1}(x_0)}{\phi_{m+1} \in H(\T), (\phi_k)_{k = 0}^{m+1} \in \cH_{m+1}} = \T.
\]
\end{itemize}
\end{lem}

\begin{proof}
Define $\cK = \ov{\set{\wt{n}}{n \in \N}}$. This set coincides with the orbit closure of $\wt{0}$ in the distal system $(H(\T)^{\Z_+},T)$. Observe that $\wt{n} \in \cH$ for any $n \in \Z$: the condition \eqref{eq:H.0} clearly holds, and for $k \geq 1$,
\begin{align*}
(\wt{n})_k(x^{k!}) &= \binom{n}{k}^{\times}(x^{k!}) = \left( \sum_{j = 1}^k s(k,j)n^j \right)^{\times} (x) = \prod_{j = 1}^k (s(k,j)n^j)^{\times} (x) \\
&= \prod_{j=1}^{k}  n^{\times \, (j)}(x^{s(k,j)}) = \prod_{j=1}^{k}  \wt{n}_1^{(j)}(x^{s(k,j)})
\end{align*}
for any $x \in \T$ and, in particular, for any $x \in \T_Q$. This proves that $\cK \subseteq \cH$.

To prove the converse, we begin by defining a direct product group $G = \prod_{k = 1}^{\infty} G_k$ where $G_1 = H(\T)$ and $G_k = \Sigma_a^B$ for all $k \geq 2$. It is compact and abelian. We wish to find a homeomorphism from $\cH$ to this group and then show that the sequence in $G$ corresponding to $(\wt{n})_{n = 1}^{\infty}$ is uniformly distributed and, a fortiori, dense. Recall that, by the Weyl criterion (Corollary 1.2 in chapter 4 of \cite{kuipers}), sequence $(x_n)$ in a compact abelian group $X$ is uniformly distributed if and only if
	\[\lim_{N \to \infty} \frac{1}{N} \sum_{n = 1}^N \chi(x_n) = 0
\]
for every non-trivial character $\chi$ of $X$.

Let $\pi_k \colon \cH \to H(\T)$ denote the projection to the $k$th coordinate for each $k \in \N$. We define functions $\Gamma_k \colon \cH \to G_k$ for $k \in \N$ as follows:
\begin{align*}
\Gamma_1 &= \pi_1, \\
\Gamma_k &= \Theta \circ (\Psi^{-1})_1 \circ \pi_k, \text{ for all } k \geq 2.
\end{align*}
(The mappings $\Theta$ and $(\Psi^{-1})_1$ were defined in the previous section.) These functions are continuous. Note also that, for any $\phi \in \cH$, $j \geq 2$ and $b \in B$,
	\[\Gamma_j(\phi)(b) = ((\Psi^{-1})_1 (\phi_j)(\exx{(b/k!)} \T_Q))_{k = 1}^{\infty} = (\phi_j(\exx{(b/k!)}))_{k = 1}^{\infty}.
\]
Let $\Gamma \colon \cH \to G$ be the continuous function $\Gamma(\phi) = (\Gamma_k(\phi))_{k=1}^{\infty} \in G$, $\phi \in \cH$. We claim that it is a homeomorphism.

To show that $\Gamma$ is injective, pick $\phi, \psi \in \cH$ and assume that $\Gamma(\phi) = \Gamma(\psi)$. Then, $\phi_0 = 1^{\times} = \psi_0$ and $\phi_1 = \psi_1$. The second identity implies that, if $x = \exx{\theta} \in \T_Q$ and $k \geq 1$, defining $y = \exx{(\theta/k!)} \in \T_Q$, we obtain
\begin{align*}
\phi_k(x) &= \phi_k(y^{k!}) = \prod_{j=1}^k \phi_1^{(j)}(y^{s(k,j)}) = \prod_{j=1}^k \psi_1^{(j)}(y^{s(k,j)}) \\
&= \psi_k(y^{k!}) = \psi_k(x),
\end{align*}
so $\phi_k(x) = \psi_k(x)$ for any $x \in \T_Q$ and $k \geq 1$. Also, since $\Gamma_k(\phi) = \Gamma_k(\psi)$ for all $k \geq 2$, we have $\phi_k(\exx{(b/n!)}) = \psi_k(\exx{(b/n!)})$ for all $b \in B_1, n \in \N$ and every $k \in \N$. Since the elements $\exx{(b/n!)}$ ($b \in B_1, n \in \N$) generate $\T$, we must have $\phi_k = \psi_k$ for every $k \in \N$. Hence, $\phi = \psi$.

To prove that $\Gamma$ is surjective, take arbitrary $\tau \in H(\T)$ and $\sigma_j \in \Sigma_a^B$ for $j \geq 2$, so $(\tau,\sigma_2,\sigma_3,\ldots) \in G$. For all $k \in \N$, we define a function $\gamma_k \in H(\T_Q)$ by the rule
\begin{equation}
\gamma_k(x^{k!}) = \prod_{j=1}^k \tau^{(j)}(x^{s(k,j)}), \label{eq:gammak}
\end{equation}
$x \in \T_Q$. Note that this condition does indeed define a function on $\T_Q$ because this group is divisible. Moreover, the function $\gamma_k$ is well-defined for any $k \in \N$. To see this, suppose that $x^{k!}, y^{k!} \in \T_Q$, $x^{k!} = y^{k!}$, $x = \exx{(p_1/q)}$ and $y = \exx{(p_2/q)}$ for some $p_1,p_2 \in \Z$ and $q \in \N$. Then, there must be some $n \in \Z$ such that
	\[\frac{p_1}{q} = \frac{p_2}{q} + \frac{n}{k!}.
\]
Let $m \in \Z$ be such that $\tau(z) = z^m$ for all $z \in \T_{k!}$. Then,
\begin{align*}
\prod_{j=1}^k \tau^{(j)}(x^{s(k,j)}) &= \prod_{j=1}^k \tau^{(j)} \left( \xp{\frac{p_1}{q}s(k,j)} \right) \\
&= \prod_{j=1}^k \tau^{(j)} \left( \xp{ \left( \frac{p_2}{q} + \frac{n}{k!} \right) s(k,j)} \right) \\
&= \left[ \prod_{j=1}^k \tau^{(j)}(y^{s(k,j)}) \right] \left[ \prod_{j=1}^k \tau^{(j)} \left( \xp{\frac{1}{k!}s(k,j)} \right) \right]^n,
\end{align*}
and the second product inside the square brackets can be calculated as follows:
\begin{align*}
\prod_{j=1}^k \tau^{(j)} \left( \xp{\frac{1}{k!}s(k,j)} \right) &= \prod_{j=1}^k \xp{\frac{1}{k!}s(k,j)m^j} = \xp{\frac{1}{k!} \sum_{j=1}^k s(k,j)m^j} \\
&= \xp{\binom{m}{k}} = 1.
\end{align*}
Hence, $\gamma_k(x) = \gamma_k(y)$. It is straightforward to check that $\gamma_k \in H(\T_Q)$ for all $k \in \N$. We use these functions to define $\phi \in E(\T)^{\Z_+}$ in the following manner:
\begin{align*}
\phi_0 &= 1^{\times}, \\
\phi_1 &= \tau, \\
\phi_k &= \Psi(\Theta^{-1}(\sigma_k),\gamma_k) \text{ for all } k \geq 2.
\end{align*} 
Now, for $k \geq 2$ and $x \in \T_Q$, we have $\phi_k(x) = \gamma_k(x)$ since $\Phi_1(x) = 1 \T_Q$ and $\Phi_2(x) = x$. Thus,
	\[\phi_k(x^{k!}) = \gamma_k(x^{k!}) = \prod_{j=1}^k \phi_1^{(j)}(x^{s(k,j)})
\]
showing that $\phi \in \cH$. Obviously, $\Gamma(\phi) = (\tau,\sigma_2,\sigma_3,\ldots)$, so $\Gamma$ is surjective. We conclude that $\Gamma$ is a homeomorphism.

We can now apply the Weyl criterion to the sequence $(\Gamma(\wt{n}))_{n=1}^{\infty}$ in $G$. Pick a non-trivial character $\chi$ of $G$. It is of the form
	\[\chi (\tau, \sigma_2, \sigma_3, \ldots) = \chi_1(\tau) \chi_2(\sigma_2) \chi_3(\sigma_3) \cdots \chi_k(\sigma_k),
\]
$(\tau, \sigma_2, \sigma_3, \ldots) \in G$, where $k \geq 2$, $\chi_1$ is a character of $H(\T)$, $\chi_j$ is a character of $\Sigma_a^B$ for $2 \leq j \leq k$, and at least one of the characters is non-trivial. We can find $z \in \T$ with $\chi_1 (\tau) = \tau (z)$ for all $\tau \in H(\T)$. Pick $\zeta \in \R$ so that $z = \exx{\zeta}$. The character group of $\Sigma_a^B$ is consists of products of characters of the form $\chi \circ P_b$ where $\chi \in \wh{\Sigma_a}$ and $P_b \colon \Sigma_a^B \to \Sigma_a$ is the projection onto the $b$-coordinate, $b \in B$. Thus, for $\sigma \in \Sigma_a^B$, we may express $\chi_j$, $2 \leq j \leq k$, as
	\[\chi_j(\sigma) = \prod_{l=1}^{L} J \left( \frac{n(j,l)}{m!} \right)(\sigma(b(j,l))_m) = \prod_{l=1}^{L} \sigma(b(j,l))_m^{n(j,l)}
\]
for some $n(j,l) \in \Z$, $m, L \in \N$ and $b(j,l) \in B$ for $1 \leq l \leq L$ and $2 \leq j \leq k$. Note that we may indeed assume $L$ to be independent of $j$ and $m$ to be independent of both $j$ and $l$ by adjusting the numbers $n(j,l)$ appropriately and choosing new numbers $b(j,l)$ if necessary together with corresponding values $n(j,l) = 0$. We may also assume that, for any $2 \leq j \leq k$, the numbers $b(j,l)$ are distinct for different values of $l$, and we put
	\[\theta_j = \sum_{l=1}^{L} \frac{n(j,l)}{m!}b(j,l).
\]
Now, for any $n \in \N$,
\begin{align*}
\chi(\Gamma(\wt{n})) &= \prod_{j=1}^k \chi_j(\Gamma_j(\wt{n})) = \binom{n}{1}^{\times}(z) \prod_{j=2}^k \prod_{l=1}^L \left( \binom{n}{j}^{\times} \left( \xp{\frac{b(j,l)}{m!}} \right) \right)^{n(j,l)} \\
&= z^n \prod_{j=2}^k \prod_{l=1}^L \xp{\frac{b(j,l)}{m!}n(j,l)\binom{n}{j}} \\
&= \xp{ \zeta n +  \sum_{j=2}^k \sum_{l=1}^L \frac{n(j,l)}{m!}b(j,l)\binom{n}{j} } \\
&= \xp{\zeta n +  \sum_{j=2}^k \theta_j \binom{n}{j}}.
\end{align*}

If at least one of the characters $\chi_j$, $2 \leq j \leq k$, is non-trivial, we may assume that $\chi_k$ is non-trivial. This means that at least one of the characters $J(n(k,l)/m!)$, $1 \leq l \leq L$, is non-trivial or, equivalently, $n(k,l) \neq 0$ for some $l$. By the linear independence of the set $\{1,b(k,1),b(k,2),\ldots,b(k,L)\}$ over the field $\Q$, we must have $\theta_k \in \R \setminus \Q$. Thus, the leading coefficient $\theta_k/k!$ of the polynomial $\zeta n +  \sum_{j=2}^k \theta_j \binom{n}{j}$ is irrational, and
	\[\lim_{N \to \infty} \frac{1}{N} \sum_{n=1}^N \chi(\Gamma(\wt{n})) = \lim_{N \to \infty} \frac{1}{N} \sum_{n=1}^N \exp \left( \zeta n +  \sum_{j=2}^k \theta_j \binom{n}{j} \right) = 0
\]
by Theorem~\ref{sec:poly}.

Suppose, then, that $\chi_j = 1$ for all $2 \leq j \leq k$. In this case, we must have $\chi_1 \neq 1$, that is, $z \neq 1$, and
	\[\lim_{N \to \infty} \frac{1}{N} \sum_{n=1}^N \chi(\Gamma(\wt{n})) = \lim_{N \to \infty} \frac{1}{N} \sum_{n=1}^N z^n = 0
\]
by Theorem~\ref{sec:poly}. This shows that $(\Gamma(\wt{n}))_{n=1}^{\infty}$ is uniformly distributed in $G$. In conclusion, $\cH = \cK$.

We now prove the second claim. Suppose that $m \in \N$ and $\phi \in \cH_m$. Let $\tau = \phi_1$, and pick $\sigma_k \in \Sigma_a^B$ for all $k \geq 2$ so that, if $m \geq 2$, $\Gamma_k(\phi_k) = \sigma_k$ for $2 \leq k \leq m$, and arbitrarily for $k > m$. Let $\psi = \Gamma^{-1}(\tau,\sigma_2,\sigma_3,\ldots) \in \cH$. We can define $\gamma_k \in H(\T_Q)$ for all $k \in \N$ with $\tau$ as in \eqref{eq:gammak}, so $\psi_k = \Psi(\Theta^{-1}(\sigma_k),\gamma_k)$ for $k \geq 2$. By assumption, $\gamma_k(x) = \phi_k(x)$ for all $x \in \T_Q$ and $2 \leq k \leq m$. Now, $\psi_1 = \tau = \phi_1$. When $2 \leq k \leq m$ (assuming that $m \geq 2$), we get $\psi_k(x^{k!}) = \gamma_k(x^{k!}) = \phi_k(x^{k!})$ for all $x \in \T_Q$ and $\psi_k(\exx{(b/n!)}) = \sigma_k(b)_n = \phi_k(\exx{(b/n!)})$ for all $b \in B$ and $n \in \N$. These identities imply that $\phi_k = \psi_k$ for $2 \leq k \leq m$, completing the proof.

The third statement follows by modifying the arguments above. Let $m \in \N$, $(\phi_k)_{k = 0}^m \in \cH_m$ and $x_0 \in \T \setminus \T_Q$ be given. For simplicity, we may assume that $x_0 = e^{i2\pi b_0}$ for some $b_0 \in B$. Let $y \in \T$ be arbitrary, and let $\theta \in \R$ be such that $y = e^{i2\pi \theta}$. Again, we choose $\tau \in H(\T)$ and $\sigma_k \in \Sigma_a^B$ for $k \in \N$ as above with the exception that $\sigma_{m+1}(b_0) = (e^{i2\pi (\theta/n!)})_{n = 1}^{\infty}$. We obtain $(\phi_k)_{k = 0}^{\infty} = \Gamma^{-1}(\tau,\sigma_2,\sigma_3,\ldots) \in \cH$. Now, $(\phi_k)_{k = 0}^{m+1} \in \cH_{m+1}$, and $\phi_{m+1}(x_0) = y$. Since $y \in \T$ was chosen arbitrarily, the claim follows.
\end{proof}

We can now describe the Ellis group of the system $(\T^{\Z_+},t)$.

\begin{thm}
\label{sec:main}
The group $E(\T^{\Z_+},t)$ is topologically isomorphic to $(\cH,\star)$ where the group operation $\star$ is defined by $\phi \star \psi = ((\phi \star \psi)_k)_{k=0}^{\infty}$,
	\[(\phi \star \psi)_k = \prod_{j=0}^k \phi_{k-j} \circ \psi_j,
\]
for all $\phi, \psi \in \cH$. In addition, the system $(E(\T^{\Z_+},t),\lambda_t)$ is isomorphic to $(\cH,T)$. The members of $E(\T^{\Z_+},t)$ are functions of the form
	\[(x_0,x_1,x_2,\ldots) \mapsto (x_0,\phi_1(x_0)x_1,\phi_2(x_0)\phi_1(x_1)x_2,\ldots) : \T^{\Z_+} \to \T^{\Z_+}
\]
where the homomorphisms $\phi_k \in H(\T)$, $k \in \N$, satisfy \eqref{eq:H.inf}.
\end{thm}

\begin{proof}
We define a mapping $\cE \colon \cH \to (\T^{\Z_+})^{\T^{\Z_+}}$ by
	\[\cE(\phi)(x) = (\phi_0(x_0),\phi_1(x_0)\phi_0(x_1),\phi_2(x_0)\phi_1(x_1)\phi_0(x_2),\ldots)
\]
for every $\phi \in \cH$ and $x \in \T^{\Z_+}$, that is,
	\[\cE(\phi)(x)_k = \prod_{j=0}^k \phi_{k-j}(x_j)
\]
for all $k \in \Z_+$. This is to be the desired group isomorphism from $\cH$ to $E(\T^{\Z_+},t)$. By Proposition~\ref{sec:iter}, $\cE(\wt{n}) = t^n \in E(\T^{\Z_+},t)$ for every $n \in \Z$, and $\cE$ is evidently continuous.

To check that $\cE$ is injective, suppose that $\cE(\phi) = \cE(\psi)$ for some $\phi, \psi \in \cH$. This entails that
\begin{equation}
\label{eq:r}
\prod_{j=0}^k \phi_{k-j}(x_j) = \prod_{j=0}^k \psi_{k-j}(x_j)
\end{equation}
for every $k \in \N$ and $x_j \in \T$, $0 \leq j \leq k$. For $k = 0$, the identity~\eqref{eq:r} reduces to $x_0 = x_0$, as $\phi_0 = 1^{\times} =\psi_0$. We proceed by induction. Suppose that $\phi_j = \psi_j$ for $0 \leq j \leq k$ for some $k \in \N$. Replacing $k$ by $k+1$ in~\eqref{eq:r} and dividing away the terms $\phi_{k+1-j}(x_j) = \psi_{k+1-j}(x_j)$, $1 \leq j \leq k+1$, we are left with $\phi_{k+1}(x_0) = \psi_{k+1}(x_0)$ for all $x_0 \in \T$. Therefore, $\phi_{k+1} = \psi_{k+1}$, and $\phi = \psi$.

By the compactness of $\cH$, the mapping $\cE$ is a homeomorphism from $\cH$ to $\cE(\cH)$. Observe that, since $(\T^{\Z_+},t)$ is distal, the positive semi-orbits of $(E(\T^{\Z_+}),\lambda_t)$ are dense, so $E(\T^{\Z_+})$ is the closure of the set $\set{t^n}{n \in \N}$. Since $\cH$ is the closure of the set $\set{\wt{n}}{n \in \N}$  (Lemma~\ref{sec:horbit}), we must have $\cE(\cH) = E(\T^{\Z_+},t)$.

We can bring the group structure from $E(\T^{\Z_+},t)$ to $\cH$ via $\cE$. That is, we may define a group operation $\star \colon \cH \times \cH \to \cH$ by $\phi \star \psi = \cE^{-1}(\cE(\phi) \cE(\psi))$ for all $\phi, \psi \in \cH$. This is the unique right topological group operation on $\cH$ for which $\cE$ is a topological isomorphism. We must find an explicit formula for $\star$. Let $\phi, \psi \in \cH$. We compute $(\cE(\phi) \cE(\psi))(x)_k$ for arbitrary $x \in \T^{\Z_+}$ and $k \in \N$:
\begin{align*}
(\cE(\phi) \cE(\psi))(x)_k &= \prod_{l=0}^k \phi_{k-l}(\cE(\psi)(x)_l) = \prod_{l=0}^k \phi_{k-l} \left( \prod_{j=0}^l \psi_{l-j}(x_j) \right) \\
&= \prod_{l=0}^k \prod_{j=0}^l \phi_{k-l} \circ \psi_{l-j}(x_j) = \prod_{j=0}^k \prod_{l=j}^k  \phi_{k-l} \circ \psi_{l-j}(x_j) \\
&= \prod_{j=0}^k \prod_{l=0}^{k-j}  \phi_{k-j-l} \circ \psi_l(x_j) = \prod_{j=0}^k \xi_{k-j}(x_j),
\end{align*}
where $\xi_j(x) = \prod_{l=0}^j \phi_{j-l} \circ \psi_l (x)$ for all $x \in \T$ and $j \in \Z_+$. Let $\xi = (\xi_k)_{k=0}^{\infty} \in E(\T)^{\Z_+}$. To make sure that $\phi \star \psi = \xi$, we only need to check that $\xi \in \cH$. Suppose that $\phi = \lim_\lambda \wt{m_\lambda}$ and $\psi = \lim_\lambda \wt{n_\lambda}$ for some nets $(m_\lambda)$ and $(n_\lambda)$ in $\Z$. Clearly, $\xi_0 = 1^{\times}$. For any $k \in \N$ and $x \in \T_Q$, the joint continuity of composition in $H(\T_Q)$ and Vandermonde's identity (V.i.) yield
\begin{align*}
\xi_k(x) &= \prod_{j=0}^k \phi_{k-j} \circ \psi_j (x) = \prod_{j=0}^k \left( \lim_\lambda \binom{m_\lambda}{k-j}^{\times} \right) \circ \left( \lim_{\lambda^{\prime}} \binom{n_{\lambda^{\prime}}}{j}^{\times} \right) (x) \\
&= \lim_\lambda \prod_{j=0}^k \left( \binom{m_\lambda}{k-j}\binom{n_\lambda}{j} \right)^{\times}(x) = \lim_\lambda \left( \sum_{j=0}^k \binom{m_\lambda}{k-j}\binom{n_\lambda}{n} \right)^{\times}(x) \\
&\stackrel{V.i.}{=} \lim_\lambda \binom{m_\lambda + n_\lambda}{k}^{\times}(x) = \lim_\lambda \left( \wt{m_\lambda + n_\lambda} \right)_k(x).
\end{align*}
If $\delta \in \cH$ is some cluster point of the net $(\wt{m_\lambda + n_\lambda})$, then $\xi_k(x) = \delta_k(x)$ for every $k \in \Z_+$ and $x \in \T_Q$. Hence, $\xi$ satisfies \eqref{eq:H.inf}, and $\xi \in \cH$.

Lastly, the mapping $\cE$ is also an isomorphism of dynamical systems since $\cE(ta) = \cE(t) \star \cE(a) = T(\cE(a))$ for all $a \in E(\T^{\Z_+},t)$.
\end{proof}

Next, we find the topological centre of $E(\T^{\Z_+},t)$. The reader should remember that topological centres of various Ellis groups were already studied in \cite{jabnam10}, but our approach is less arduous thanks to the explicit description of $\cH$.

\begin{thm}
\label{sec:topcentr}
The topological centre of $E(\T^{\Z_+},t)$ is $\set{t^n}{n \in \Z}$ or, equivalently, $\Lambda(\cH) = \set{\wt{n}}{n \in \Z}$.
\end{thm}

\begin{proof}
We must show that $\Lambda(\cH) \subseteq \set{\wt{n}}{n \in \Z}$, so pick $\phi \in \Lambda(\cH)$. If $(\psi_\lambda)$ is a net in $\cH$ converging to some $\psi$, we must have $\phi \star \psi_\lambda \conv{\lambda} \phi \star \psi$, that is,
\begin{equation}
\label{centr}
\prod_{j=0}^k \phi_{k-j} \circ (\psi_\lambda)_j(x) \conv{\lambda} \prod_{j=0}^k \phi_{k-j} \circ \psi_j(x) 
\end{equation}
for every $k \in \Z_+$ and $x \in \T$. For $k=0$, both sides are equal to $x$ for all $\lambda$ by \eqref{eq:H.0}. For $k=1$, we get $\phi_{1}(x) (\psi_\lambda)_1(x) \conv{\lambda} \phi_{1}(x) \psi_1(x)$, $x \in \T$, which is true whenever $\psi_\lambda \conv{\lambda} \psi$ in $\cH$. For $k=2$, the condition reduces to $\phi_1((\psi_\lambda)_1(x)) \conv{\lambda} \phi_1(\psi_1(x))$ for all $x \in \T$ and whenever $\psi_\lambda \conv{\lambda} \psi$ in $\cH$. This condition is strong enough to guarantee that $\phi_1$ is continuous. We argue by contradiction. Suppose that $\phi_1$ is not continuous, and take a net $(x_\lambda)$ in $\T$ converging to some $x \in \T$ such that $\phi_1 (x_\lambda)$ does not converge to $\phi_1(x)$. By compactness, we may assume that $\phi_1(x_\lambda) \conv{\lambda} z$ for some $z \in \T$, $z \neq \phi_1(x)$. Let $y \in \T \setminus \T_Q$ be arbitrary. Pick $\psi^{\prime}_\lambda \in H(\T)$ for every $\lambda$ so that $\psi^{\prime}_\lambda(y) = x_\lambda$. According to Lemma~\ref{sec:horbit}, we can find $\psi_\lambda \in \cH$ so that $(\psi_\lambda)_1 = \psi^{\prime}_\lambda$ for every $\lambda$. Passing to a subnet if necessary, we may assume that $(\psi_\lambda)$ converges to some $\psi \in \cH$. Then, $(\psi_\lambda)_1(y) = x_\lambda$ for every $\lambda$, but
	\[\phi_1 (x_\lambda) = \phi_1 ((\psi_\lambda)_1(y)) \conv{\lambda} \phi_1(\psi_1(y)) = \phi_1(x),
\]
contradicting the fact that $\phi_1(x_\lambda)$ converges to $z \neq \phi_1(x)$. Hence, $\phi_1$ is continuous. This means that $\phi_1 = m^{\times}$ for some $m \in \Z$.

We can show that $\phi = \wt{m}$ by induction. The identities $\phi_0 = \wt{m}_0$ and $\phi_1 = \wt{m}_1$ have already been established. Suppose that $\phi_j = \wt{m}_j$ for $0 \leq j \leq k$ for some $k \in \N$. We claim that $\phi_{k+1} = \binom{m}{k+1}^{\times}$. By replacing $k$ with $k+2$ in~\eqref{centr} and dividing away the terms $\phi_{k+2}(x)$ and the terms involving $\phi_j$ for $0 \leq j \leq k$ from both sides, we get
	\[\phi_{k+1}((\psi_\lambda)_1(x)) \stackrel{\lambda}{\longrightarrow} \phi_{k+1}(\psi_1(x))
\]
for all $x \in \T$ and all nets $(\psi_\lambda)$ converging to $\psi$ in $\cH$. Again, this means that $\phi_{k+1}$ is continuous, so $\phi_{k+1} = p^{\times}$ for some $p \in \Z$. The number $p$ must be such that, for any $x \in \T_Q$,
\begin{align*}
x^{(k+1)!p} &= \phi_{k+1}(x^{(k+1)!}) = \prod_{j=1}^{k+1} \phi_1^{(j)}(x^{s(k+1,j)}) = \prod_{j=1}^{k+1} (s(k+1,j)m^j)^{\times}(x) \\
&= \left( \sum_{j=1}^{k+1} s(k+1,j) m^j \right)^{\times}(x) = \left( (k+1)!\binom{m}{k+1} \right)^{\times}(x).
\end{align*}
We infer that, for any $\theta \in \Q$, there is some $n_\theta \in \Z$ such that
	\[\theta (k+1)! p = \theta (k+1)!\binom{m}{k+1} + n_\theta.
\]
If $q = p - \binom{m}{k+1} \neq 0$, then by choosing $\theta = 1/(2(k+1)!q)$ in the identity above we get $n_\theta = 1/2$, which is not possible. Hence, $p = \binom{m}{k+1}$, as desired. This shows that $\phi = \wt{m}$, and the proof is complete.
\end{proof}

The Ellis groups of the basic systems can be derived from $E(\T^{\Z_+},t)$. We start with the infinite-dimensional case:

\begin{thm}
\label{sec:infdim}
Let $(\T^{\N},s)$ be the $(\infty,x_0)$-system for some $x_0 \in \T$. The group $E(\T^{\N},s)$ is topologically isomorphic to $(\cH,\star)$, and $\Lambda(E(\T^{\N},s)) = \set{s^n}{n \in \Z}$. The members of $E(\T^{\N},s)$ are mappings of the form
	\[(x_1,x_2,\ldots) \mapsto (\phi_1(x_0)x_1,\phi_2(x_0)\phi_1(x_1)x_2,\ldots) : \T^{\N} \to \T^{\N}
\]
where the sequence $(\phi_k)_{k = 1}^{\infty} \in H(\T)^{\N}$ satisfies the condition \eqref{eq:H.inf}.
\end{thm}

\begin{proof}
First, we identify $(\T^{\N},s)$ with a subsystem of $(\T^{\Z_+},t)$. Let $Y \subseteq \T^{\Z_+}$ be the set
	\[Y = \set{y \in \T^{\Z_+}}{y_0 = x_0}.
\]
It is clearly a non-empty, closed set with $t(Y) = Y$, and we obtain a subsystem $(Y,t)$ of $(\T^{\Z_+},t)$. Let $\pi \colon Y \to \T^{\N}$ be the mapping defined by $\pi(y)_k = y_k$ for all $y \in Y$ and $k \in \N$. It is clearly an isomorphism. Hence, the induced homomorphism $\wt{\pi} \colon E(Y,t) \to E(\T^{\N},s)$ is also an isomorphism.

Let $\kappa \colon E(\T^{\Z_+},t) \to E(Y,t)$ be the mapping $\kappa(a) = a|_Y$, $a \in E(\T^{\Z_+},t)$. It is obviously a continuous, surjective homomorphism of semigroups. We claim that it is injective. To see this, pick $a,b \in E(T^{\Z_+},t)$, and suppose that $\kappa(a) = \kappa(b)$. Let $\phi, \psi \in \cH$ be such that $\cE(\phi) = a$ and $\cE(\psi) = b$. In other words,
	\[\prod_{j=0}^k \phi_{k-j}(y_j) = \cE(\phi)(y)_k = \cE(\psi)(y)_k = \prod_{j=0}^k \psi_{k-j}(y_j)
\]
for all $k \in \N$ and $y \in Y$. We have $\phi_0 = 1^{\times} = \psi_0$. If $\phi_j = \psi_j$ for all $0 \leq j \leq k$ for some $k \in \Z_+$, then
\begin{align*}
\phi_{k+2}(x_0) \phi_{k+1}(y_1) \prod_{j=2}^{k+2} \phi_{k+2-j}(y_j) &= \cE(\phi)(y)_{k+2} = \cE(\psi)(y)_{k+2} \\
&= \psi_{k+2}(x_0) \psi_{k+1}(y_1) \prod_{j=2}^{k+2} \psi_{k+2-j}(y_j) \\
&= \psi_{k+2}(x_0) \psi_{k+1}(y_1) \prod_{j=2}^{k+2} \phi_{k+2-j}(y_j)
\end{align*}
for all $y \in Y$. Thus, $\phi_{k+2}(x_0) \phi_{k+1}(y_1) = \psi_{k+2}(x_0) \psi_{k+1}(y_1)$ for all $y_1 \in \T$. Choosing $y_1 = 1$, we get $\phi_{k+2}(x_0) = \psi_{k+2}(x_0)$, so $\phi_{k+1} = \psi_{k+1}$. By induction, $\phi = \psi$, as desired.

Now, $\wt{\pi} \circ \kappa \colon E(\T^{\Z_+},t) \to E(\T^{\N},s)$ is a topological isomorphism of semigroups. By Theorem~\ref{sec:main}, $E(\T^{\N},s)$ is topologically isomorphic to $(\cH,\star)$. It also follows from the same theorem that $E(\T^{\N},s)$ consists of those mappings that are of the stated form.
\end{proof}

Note that the choice of $x_0$ does not affect the structure of the Ellis group. The finite-dimensional case is different in this regard. In the theorem below, the key parts concerning the Ellis groups appear also in \cite{pikula10}, and the topological centre in case (iii) is treated in \cite{jabnam10}. The proof is based on the infinite-dimensional case, so we sketch it for its unifying nature.

\begin{thm}
\label{sec:findim}
Let $m \in \N$ and $x_0 \in \T$, and let $(\T^m,r)$ be the $(m,x_0)$-system.
\begin{itemize}
\item[(i)] If $m = 1$, then $E(\T,r)$ is topologically isomorphic to the subgroup $G = \ov{\set{x_0^n}{n \in \Z}}$ of $\T$, and $\Lambda(E(\T,r)) = E(\T,r)$. The members of $E(\T,r)$ are translations by the members of $G$.
\item[(ii)] If $m \geq 2$ and $x_0 \in \T_Q$, then the group $E(\T^m,r)$ is topologically isomorphic to $(\cH_{m-1},\star)$ where the group operation $\star$ on $\cH_{m-1}$ is given by $\phi \star \psi = ((\phi \star \psi)_k)_{k=0}^{m-1}$,
	\[(\phi \star \psi)_k = \prod_{j=0}^k \phi_{k-j} \circ \psi_j
\]
for all $\phi, \psi \in \cH_{m-1}$. In addition, $\Lambda(\cH_{m-1}) = \set{\rho_{m-1}(\wt{n})}{n \in \Z}$.
\item[(iii)] If $m \geq 2$ and $x_0 \in \T \setminus \T_Q$, then the group $E(\T^m,r)$ is topologically isomorphic to $(\cH_{m-1} \times \T,\ast)$ where the group operation $\ast$ on $\cH_{m-1} \times \T$ is defined by
	\[(\phi,x) \ast (\psi,y) = (\phi \star \psi, xy \prod_{k=1}^{m-1} \phi_{m-k} \circ \psi_k (x_0)),
\]
$(\phi,x), (\psi,y) \in \cH_{m-1} \times \T$. In addition,
	\[\Lambda(\cH_{m-1} \times \T) = \set{(\rho_{m-1}(\wt{n}),x)}{n \in \Z, x \in \T}.
\]
\end{itemize}
\end{thm}

\begin{proof}
The case $m = 1$ is well known. We only prove (ii) and (iii), so suppose that $m \geq 2$. Let $\pi_m \colon \T^{\N} \to \T^m$ be the projection onto the first $m$ coordinates. It is a homomorphism from the $(\infty,x_0)$-system $(\T^{\N},s)$ to $(\T^m,r)$. Using the induced homomorphism $\wt{\pi_m} \colon E(\T^{\N},s) \to E(\T^m,r)$, we see that any $a \in E(\T^m,r)$ must be such that
\begin{equation}
a(x)_k = \prod_{j = 0}^k \phi_{k-j}(x_j)
\label{eq:coord}
\end{equation}
for all $x \in \T^m$ and for some $\phi \in \cH_m$. Conversely, given any such function $a \colon \T^m \to \T^m$ corresponding to some $\phi \in \cH_m$, the surjectivity of $\rho_m \colon \cH \to \cH_m$ (Lemma~\ref{sec:horbit}) enables us to find $\phi \in \cH$ and a corresponding $a^{\prime} \in E(\T^{\N},s)$ so that $\rho_m(\phi) = (\phi_k)_{k = 0}^m$ and $\wt{\pi_m}(a^{\prime}) = a$. This shows that $E(\T^m,r)$ consists of all functions of the type defined by \eqref{eq:coord}. Note that the function $\phi_m$ appears in this equation only when $k = m$ in the term $\phi_m(x_0)$. When $x_0 \in \T_Q$, the value $\phi_m(x_0)$ depends on $\phi_1$ due to the restriction \eqref{eq:H.m}. When $x_0 \in \T \setminus \T_Q$, $\phi_m(x_0)$ can attain any value independently of the functions $\phi_k$, $1 \leq k \leq m-1$ by (iii) in Lemma~\ref{sec:horbit}.

In the case $x_0 \in \T_Q$, pick $y_0 \in \T$ so that $y_0^{m!} = x_0$. The desired topological isomorphism is $\cE^{\prime} \colon \cH_{m-1} \to E(\T^m,r)$,
\begin{align*}
\cE^{\prime}(\phi)(x)_k &= \prod_{j=0}^k \phi_{k-j}(x_j) \text{ for all } 1 \leq k \leq m-1, \\
\cE^{\prime}(\phi)(x)_m &= \prod_{j=1}^m \phi_1^{(j)}(y_0^{s(m,j)}) \prod_{l=1}^m \phi_{m-l}(x_l)
\end{align*}
for all $\phi \in \cH_m$ and $x  \in \T^m$. Given $\phi \in \cH_{m-1}$, we can extend this sequence to some $(\phi_k)_{k = 0}^m \in \cH_m$ by using the surjectivity of $\rho_{m-1}$, and the mapping $\cE^{\prime}(\phi) \colon \T^m \to \T^m$ is seen to satisfy \eqref{eq:coord}. This proves that $\cE^{\prime}(\phi) \in E(\T^m,r)$ for every $\phi \in \cH_{m-1}$ and that $\cE^{\prime}$ is surjective. The mapping $\cE^{\prime}$ is clearly continuous. To prove that it is injective, one can argue by induction as in the case of $\cE$ in the proof of Theorem~\ref{sec:main}. The group operation $\star$ making $\cE^{\prime}$ an isomorphism of groups is also easy to find by following the example of the proof of Theorem~\ref{sec:main}.

Suppose that $x_0 \in \T \setminus \T_Q$. In this case, we can define $\cE^{\prime \prime} \colon \cH_{m-1} \times \T \to E(\T^m,r)$ by
\begin{align*}
\cE^{\prime \prime}(\phi, y)(x)_k &= \prod_{j=0}^k \phi_{k-j}(x_j) \text{ for all } 1 \leq k \leq m-1, \\
\cE^{\prime \prime}(\phi, y)(x)_m &= y \prod_{k=1}^m \phi_{m-k}(x_k)
\end{align*}
for all $\phi \in \cH_{m-1}$, $y \in \T$ and $x \in \T^m$. Again, the range of $\cE^{\prime \prime}$ is the whole group $E(\T^m,r)$, and $\cE^{\prime \prime}$ is continuous. Its injectivity is proved as in the previous case, and finding the group operation $\ast$ is straightforward.

In either case, the topological centre of $E(\T^m,r)$ can be found by modifying the arguments of Theorem~\ref{sec:topcentr}. The inductive process is similar, but it stops after finitely many steps. The last step must be handled separately, taking into account whether $x_0 \in \T_Q$ or $x_0 \in \T \setminus \T_Q$.
\end{proof}

\begin{rems}
As mentioned in the introduction, the case of $(2,x_0)$-systems for $x_0 \in \T \setminus \T_Q$ was proved by Namioka in \cite{namioka82}. The techniques we have applied in this section are inspired by those found in the aforementioned source. Most importantly, statements (i) and (iii) of Lemma~\ref{sec:horbit} generalise Lemma 2 in \cite{namioka82}. The explicit description of $\cH$ (and of $\cH_m$) is the main contribution of this section to the work begun by Brown (\cite{brown69}), Hahn (\cite{hahn65}), Jabbari and Namioka (\cite{jabnam10}), and Piku{\l}a (\cite{pikula10}). The cited articles, except the first, refer to the sets $\cH$ or $\cH_m$ or to some related constructions as closures of collections of functions, but none of them describe the sets in conrete terms.

The group operation $\star$ on $\cH_m$, $m \in \N$, is obviously analogous to the group operation on $\cH$ denoted by the same symbol, and the projection $\rho_m \colon \cH \to \cH_m$ is a group homomorphism.
\end{rems}

We shall encounter situations involving inverses of elements of $\cH$ with respect to the operation $\star$, but using an explicit formula for the inverse can be avoided. The proposition below is one example of this.

\begin{prop}
\label{sec:commutator}
Let $\phi,\psi \in \cH$, and suppose that, for some $k \in \Z_+$, $\phi_j = \wt{0}_j$ for all $0 \leq j \leq k$. Then,
\begin{align*}
(\phi^{-1} \star \psi^{-1} \star \phi \star \psi)_j &= 0^{\times} \text{ for all } 1 \leq j \leq k + 1, \\
(\phi^{-1} \star \psi^{-1} \star \phi \star \psi)_{k+2} &= \ov{(\psi_1 \circ \phi_{k+1})} (\phi_{k+1} \circ \psi_1).
\end{align*}
\end{prop}

\begin{proof}
Pick $\xi \in \cH$ so that $\xi_j = 0^{\times}$ for $1 \leq j \leq k+1$ and $\xi_{k+2} = \ov{(\psi_1 \circ \phi_{k+1})} (\phi_{k+1} \circ \psi_1)$. (The condition \eqref{eq:H.inf} is indeed satisfied for $k+2$ in place of $k$.) The claim follows once it is shown that $(\phi \star \psi)_j = (\psi \star \phi \star \xi)_j$ for all $1 \leq j \leq k+2$ since then $\rho_{k+2}(\phi \star \psi) = \rho_{k+2}(\psi \star \phi \star \xi)$, and $\rho_{k+2}(\xi) = \rho_{k+2}(\phi^{-1} \star \psi^{-1} \star \phi \star \psi)$. This is a simple matter of computing both terms, and it is left to the reader.
\end{proof}

As an easy corollary, we see that the group $\cH_m$ is nilpotent for any $m \in \N$. It has a central series consisting of the following closed subgroups:
	\[\cH_{m,n} = \set{\phi \in \cH_m}{ \phi_k = \wt{0}_k \text{ for all } 0 \leq k \leq m-n},
\]
$0 \leq n \leq m$. Also, the Ellis group of a minimal $(m,x_0)$-systems for $m \in \N$ is nilpotent, as is shown in \cite{pikula10} with arguments that avoid a concrete explication of $\cH_m$ or $\cH_{m-1}$. Nilpotent Ellis groups occur also in the case of the so-called nil-systems or nil-transformations, which were studied by Glasner in \cite{glasner93}, a work inspired by \cite{namioka82}, and more recently by Donoso in \cite{donoso14}, where the general structure of the Ellis groups of the nil-systems was analysed with the dynamic cube technique of Host, Kra and Maass (\cite{hkm10}). The basic systems are included in the class of nil-systems.


\section{On systems with quasi-discrete spectrum}

In this section and the next, we shall use the basic systems and their Ellis groups to study systems with quasi-discrete spectrum. Most importantly, we shall identify $(\cH,T)$ as the universal minimal system with quasi-discrete spectrum (Corollary~\ref{sec:universal}).

First, we fix some notations and recall the necessary definitions. Consider a dynamical system $(X,s)$. Let $\cC(X,\T)$ denote the multiplicative subgroup of $\cC(X)$ of all continuous functions from $X$ to $\T$. If $G \subseteq \cC(X,\T)$ is a subgroup, we define $G^s \subseteq \cC(X,\T)$ by
	\[G^s = \set{g \in \cC(X,\T)}{s^{\ast}g = fg \text{ for some } f \in G}.
\]
(Recall that $s^{\ast} \colon \cC(X) \to \cC(X)$ is the adjoint of $s$.) It is easy to see that $G^s$ is also a subgroup of $\cC(X,\T)$. We define a sequence $(G_m(X,s))_{m=0}^{\infty}$ of subgroups of $\cC(X,\T)$ by letting $G_0(X,s)$ be the group of constant functions from $X$ to $\T$ and by defining $G_{m+1}(X,s) = G_m(X,s)^s$ for all $m \in \Z_+$. An inductive argument shows the sequence to be increasing, i.e., $G_m(X,s) \subseteq G_{m+1}(X,s)$ for all $m \in \Z_+$, so we can define a group $G(X,s) = \bigcup_{m = 0}^{\infty} G_m(X,s)$. The functions of $G_1(X,s)$ are called \emph{eigenfunctions} of $(X,s)$, and members of $G(X,s)$ are called \emph{quasi-eigenfunctions} of $(X,s)$. We say that $(X,s)$ has (topological) \emph{quasi-discrete spectrum} if the linear span of $G(X,s)$ is dense in $\cC(X)$ or, equivalently, if the quasi-eigenfunctions separate points. It is not hard to see that the property of having quasi-discrete spectrum is preserved under isomorphisms.

\begin{prop}
\label{sec:qeffactor}
Let $(X,s)$ and $(Y,t)$ be dynamical systems, and let $\pi \colon X \to Y$ be a homomorphism. If $g \in G_m(X,s)$, $m \in \Z_+$, is such that $g(x) = g(x^{\prime})$ whenever $x,x^{\prime} \in X$ satisfy $\pi(x) = \pi(x^{\prime})$, then $g = \pi^{\ast} f$ for some $f \in G_m(Y,t)$.
\end{prop}

\begin{proof}
The proof is done by induction on $m$. The case $m = 0$ is clear. Assuming that the claim of the proposition holds for some $m \in \Z_+$, pick a function $g \in G_{m+1}(X,s)$ that is constant on the $\pi$-fibres, and let $f \in \cC(X,\T)$ be the function for which $g = \pi^{\ast}f$ (the continuity of $f$ follows from a compactness argument). Let $g^{\prime} \in G_m(X,s)$ be such that $s^{\ast} g = g^{\prime} g$. Now, $s^{\ast} g$ is constant on the $\pi$-fibres, so $g^{\prime} = (s^{\ast}g)\ov{g}$ has this property as well. Therefore, we can find $f^{\prime} \in G_m(Y,t)$ with $g^{\prime} = \pi^{\ast}f^{\prime}$. It is easy to check that $t^{\ast} f = f^{\prime} f$, so $f \in G_{m+1}(Y,t)$, completing the argument.
\end{proof}

Quasi-eigenfunctions are connected to basic systems. For a dynamical system $(X,s)$, if $f_m \in G_m(X,s)$ for some $m \in \N$, we can find unique functions $f_k \in G_k(X,s)$ for $0 \leq k \leq m-1$ so that $s^{\ast}f_k = f_{k-1} f_k$ for all $1 \leq k \leq m$. Let $x_0 \in \T$ be the constant value of $f_0$, and define $\pi \colon X \to \T^m$ by
	\[\pi(x) = (f_1(x),f_2(x),\ldots,f_m(x))
\]
for all $x \in X$. This function is continuous. If $r \colon \T^m \to \T^m$ is the homeomorphism for which $(\T^m,r)$ is the $(m,x_0)$-system, then $\pi \circ s = r \circ \pi$. We say that a function $\pi$ with this property is an $(m,x_0)$\emph{-representation} of $(X,s)$ or simply a \emph{basic representation} if we wish to leave the parameters $m$ and $x_0$ vague. In the terminology of Hahn and Parry, the triple $(\T^m,r,\pi)$ is a \emph{diagonal representation} of $(X,s)$. If, on the other hand, $\pi \colon X \to \T^m$ is an $(m,x_0)$-representation of $(X,s)$ for some $m \in \N$ and $x_0 \in \T$, the coordinate functions $\pi_k$ of $\pi$ satisfy $\pi_k \in G_k(X,s)$ for all $1 \leq k \leq m$ and $s^{\ast} \pi_1 = x_0 \pi_1$. This is the content of Theorem 9 in \cite{hahnparry65}. It follows that we can characterise the systems with quasi-discrete spectrum as those systems for which there exists a family of basic representations that separates points. (Note that these results are stated in \cite{hahnparry65} for \emph{minimal metric} systems, but these added restrictions are not necessary for us.) Consequently, a system with quasi-discrete spectrum is isomorphic to a subsystem of a product of basic systems, and therefore it is distal. Observe that the property of having quasi-discrete spectrum is preserved when taking products and subsystems, as is easy to show by applying the definition. The existence of a universal minimal system with quasi-discrete spectrum comes as a corollary of this fact (use Theorem 4.30 in chapter 4 of \cite{devries}).

\begin{prop}
\label{sec:qdslift}
If a system $(X,s)$ has quasi-discrete spectrum, then $(E(X,s),\lambda_s)$ has quasi-discrete spectrum.
\end{prop}

The proof of this is immediate since $(E(X,s),\lambda_s)$ can be seen as a subsystem of products of $(X,s)$. As we shall see, the converse is not true, i.e., there is a system $(X,s)$ that does not have quasi-discrete spectrum but for which $(E(X,s),\lambda_s)$ does have it.

Recall that we use the notation $l^{\infty} = l^{\infty}(\Z)$.

\begin{defns}
A function $f \in l^{\infty}$ is said to be of \emph{polynomial type} if there exists some polynomial $p \colon \Z \to \R$ so that $f(n) = \exx{p(n)}$ for all $n \in \Z$. When $m \in \Z_+$, the set of functions of polynomial type for which the corresponding polynomial has degree at most $m$ is denoted by $\cP_m$. Let $\cP = \bigcup_{m = 0}^{\infty} \cP_m$. The smallest closed linear subspace of $l^{\infty}$ containing $\cP$ is denoted by $\cW$.
\end{defns}

Note that $\cP$ is a group under multiplication, so $\cW$, the `Weyl algebra', is indeed an algebra. It contains the constants, and it is closed under complex conjugation. It is also translation invariant since $\cP$ has this property. In fact, $\cW$ is $m$-admissible and hence determines a right topological semigroup compactification of $\Z$ (see \cite{bjm} for general theory on semigroup compactifications). We deal with this in Corollary~\ref{sec:wcomp}. Observe also that the linear span of $\cP_1$ is dense in $\mathcal{AP}(\Z)$, the space of almost periodic functions on $\Z$ (which coincides with the space $\mathcal{SAP}(\Z)$ of strongly almost periodic functions on $\Z$, i.e., the closed linear span of characters of $\Z$; \cite{bjm}).

The first half of the next proposition appears in the proof of Theorem 1 in \cite{hahnparry65}.

\begin{prop}
\label{sec:GP}
Let $(X,s)$ be a minimal system, and let $x \in X$. Given $f \in \cC(X)$, we have $f \in G_m(X,s)$ for some $m \in \Z_+$ if and only if $\omega_x f \in \cP_m$.
\end{prop}

\begin{proof}
Suppose that $f \in G_m(X,s)$, $m \in \Z_+$. We can find $f_k \in G_k(X,s)$ for $0 \leq k \leq m$ so that $s^{\ast}f_k = f_{k-1} f_k$ for all $1 \leq k \leq m$ and $f = f_m$. Let $x_0 \in \T$ be the constant $x_0 = f_0$, and let $\pi \colon X \to \T^m$ be the $(m,x_0)$-representation of $(X,s)$ whose coordinate functions are $f_k$, $1 \leq k \leq m$. Let $r \colon \T^m \to \T^m$ be the homeomorphism of the $(m,x_0)$-system. Now, for all $n \in \Z$,
	\[f(s^n x) = \pi(s^n x)_m = (r^n \pi(x))_m = \prod_{k=0}^m \binom{n}{m-k}^{\times}(f_k(x)).
\]
It is clear that the product term on the right describes an element of $\cP_m$.

The converse (i.e., the claim that if $f \in \cC(X)$ and $\omega_x f \in \cP_m$, then $f \in G_m(X,s)$) is proved by induction on $m \in \Z_+$. In the case $m = 0$, any $f \in \cC(X)$ with $\omega_x f \in \cP_0$ must be constant (in $\T$) since the orbit of $x$ is dense in $X$, so $f \in G_0(X,s)$. Assuming that the claim holds for some $m \in \Z_+$, consider a function $f \in \cC(X)$ such that $\omega_x f \in \cP_{m+1}$. Let $p \colon \Z \to \R$ be the polynomial corresponding to $f$, say $p(n) = \theta_0 + \sum_{k=1}^{m+1} \theta_k n^k$, $n \in \Z$, $\theta_k \in \R$, $0 \leq k \leq m+1$. If $\theta_{m+1} = 0$, then $\omega_x f \in \cP_m$ and $f \in G_m(X,s) \subseteq G_{m+1}(X,s)$. Suppose that $\theta_{m+1} \neq 0$. It is easy to check that $n \mapsto p(n+1) - p(n) : \Z \to \R$ is polynomial whose degree is $m$, so $\omega_x((s^{\ast}f) \ov{f}) = R(\omega_x f) \ov{\omega_x f} \in \cP_m$. (Recall that $R$ denotes the shift operator on $l^{\infty}$.) Thus, $(s^{\ast}f) \ov{f} \in G_m(X,s)$, from which it follows that $f \in G_{m+1}(X,s)$.
\end{proof}

\begin{thm}
\label{sec:qdsweyl}
Let $(X,s)$ be a minimal system, and let $x \in X$. The system $(X,s)$ has quasi-discrete spectrum if and only if $\lspan(\omega_x \cC(X) \cap \cP)$ is dense in $\omega_x \cC(X)$. Therefore, if $(X,s)$ has quasi-discrete spectrum, then $\omega_x \cC(X) \subseteq \cW$.
\end{thm}

\begin{proof}
Firstly, by Proposition~\ref{sec:GP}, $\omega_x \cC(X) \cap \cP = \omega_x G(X,s)$. Secondly, the system $(X,s)$ has quasi-discrete spectrum if and only if $\lspan G(X,s)$ is dense in $\cC(X)$. Combining these facts results in the first claim. The second one follows from the first one.
\end{proof}

The converse of the inclusion statement does not hold. In the next section, we construct a minimal system $(X,s)$ with $\omega_x\cC(X) \subseteq \cW$ for some (or any) $x \in X$ but which does not have quasi-discrete spectrum. (It is the same system that serves as a counterexample to the converse of Proposition~\ref{sec:qdslift}.)

The group $G(\cH,T)$ of all quasi-eigenfunctions of $(\cH,T)$ is obtained from the character group of $H(\T)^{\Z_+}$. Let $V$ denote the weak product of $\Z_+$ copies of $\T$, i.e., $V$ consists of those $v \in \T^{\Z_+}$ for which $v_n = 1$ for all except at most finitely many $n \in \Z_+$, and $V$ is equipped with the usual pointwise product. Let $\sigma \colon V \to V$ be the shift map: $\sigma(v)_k = v_{k+1}$, $v \in V$, $k \in \Z_+$. We define $q \colon V \to \cC(\cH,\T)$ by
	\[q(v)(\phi) = \prod_{k = 0}^{\infty} \phi_k(v_k),
\]
$v \in V$, $\phi \in \cH$. This function is a homomorphism. Computing $T^{\ast}q(v)$ for an arbitrary $v \in V$ of the form $v = (v_0,v_1,\ldots,v_n,1,1,\ldots)$ yields
\begin{align*}
T^{\ast}q(v)(\phi) &= q(v)(\wt{1} \star \phi) = \prod_{k = 0}^n (\wt{1} \star \phi)_k(v_k) = v_0 \prod_{k = 1}^n \phi_{k-1}(v_k) \phi_k(v_k) \\
&= \left[ \prod_{k = 0}^{n-1} \phi_k(v_{k+1}) \right] \left[ \prod_{k = 0}^n \phi_k(v_k) \right] = q(\sigma(v)v)(\phi),
\end{align*}
$\phi \in \cH$. Hence, $T^{\ast}q(v) = q(\sigma(v)v)$ for all $v \in V$. An inductive argument shows that $q(v) \in G_m(\cH,T)$ for any $v \in V$ for which $v_n = 1$ for $n > m$, $m \in \Z_+$, and consequently, $q(V) \subseteq G(\cH,T)$. The mapping $q$ is injective. To show this, consider $v,w \in V$ such that $q(v) = q(w)$. In particular, $v_0 = q(v)(\wt{0}) = q(w)(\wt{0}) = w_0$. If $n \in \Z_+$ and $v_k = w_k$ for all $0 \leq k \leq n$, then
\begin{align*}
\left[\prod_{k = 0}^n \binom{n+1}{k}^{\times}(v_k) \right] v_{n+1} &= q(v)(\wt{n+1}) = q(w)(\wt{n+1}) \\
&= \left[\prod_{k = 0}^n \binom{n+1}{k}^{\times}(v_k) \right] w_{n+1},
\end{align*}
so $v_{n+1} = w_{n+1}$. The homomorphism $q$ maps $V$ onto $G(\cH,T)$. Firstly, any constant function $c \in \cC(\cH,\T)$ can be written as $c = q(c,1,1,\ldots)$, so $\cG_0(\cH,T) \subseteq q(V)$. Assuming that $G_m(\cH,T) \subseteq V$ for some $m \in \Z_+$, consider a quasi-eigenfunction $g \in G_{m+1}(\cH,T)$ with $T^{\ast}g = fg$ for some $g \in G_m(\cH,T)$. Let $v \in V$ be such that $f = q(v)$, and define $w \in V$ so that $\sigma(w) = v$ and $w_0 = g(\wt{0})$. Now, $q(w) \in G_{m+1}(\cH,T)$ has the property $T^{\ast}q(w) = q(v)q(w) = fq(w)$, so $g\ov{q(w)} \in G_{m+1}(\cH,T)$ satisfies $T^{\ast}(g\ov{q(w)}) = g\ov{q(w)}$. Since $(\cH,T)$ is minimal, the function $g\ov{q(w)}$ must be constant. We have $(g\ov{q(w)})(\wt{0}) = 1$, so $g = q(w)$, proving that $G_{m+1}(\cH,T) \subseteq q(V)$.

Note that the family $q(V)$ separates the points of $\cH$ since it consists of restrictions of characters of $H(\T)^{\Z_+}$. The following theorem is now proved:

\begin{thm}
The mapping $q \colon V \to G(\cH,T)$ is a group isomorphism, and $(\cH,T)$ has quasi-discrete spectrum.
\end{thm}

The representation of $(\cH,T)$ as a subalgebra of $l^{\infty}$ is the Weyl algebra:

\begin{thm}
\label{sec:hweyl}
For the system $(\cH,T)$, we have $\omega_{\wt{0}} \cC(\cH) = \cW$.
\end{thm}

\begin{proof}
We know that $\omega_{\wt{0}} \cC(\cH) \subseteq \cW$ since $(\cH,T)$ has quasi-discrete spectrum. On the other hand, $\cP_m \subseteq \omega_{\wt{0}} \cC(\cH)$ for every $m \in \Z_+$, and we can prove this by induction on $m$. The case $m = 0$ is clear. Assume that $\cP_m \subseteq \omega_{\wt{0}} \cC(\cH)$ for some $m \in \Z_+$, and let $f \in \cP_{m+1}$. Suppose that $f(n) = \exx{a(n)}$, $n \in \Z$, where $a \colon \Z \to \R$ is a polynomial of the form $a(n) = \theta_0 + \sum_{k=1}^{m+1} \theta_k n^k$, $n \in \Z$, where $\theta_k \in \R$, $0 \leq k \leq m+1$. Let $b \colon \Z \to \R$ be the polynomial $b(n) = (m+1)! \theta_{m+1} \binom{n}{m+1}$, $n \in \Z$, and define $g \in \cP_{m+1}$ by $g(n) = \exx{b(n)}$, $n \in \Z$. Then, $f \ov{g} \in \cP_m \subseteq \omega_{\wt{0}} \cC(\cH)$. The function $g$ can be expressed as $g = \omega_{\wt{0}} q(v)$, the element $v \in V$ being chosen so that $v_{m+1} = \exx{(m+1)!\theta_{m+1}}$ and $v_k = 1$ for $k \neq m+1$, so $g \in \omega_{\wt{0}} \cC(\cH)$. Hence, also $f \in \omega_{\wt{0}} \cC(\cH)$, as desired. We conclude that $\cW \subseteq \omega_{\wt{0}} \cC(X)$.
\end{proof}

\begin{cor}
\label{sec:wcomp}
Let $\tau \colon \Z \to \cH$ be the group isomorphism $\tau(n) = \wt{n}$, $n \in \Z$. Then, $(\tau,\cH)$ is a $\cW$-compactification of $\Z$.
\end{cor}

The proof follows from noting that $\tau^{\ast} \cC(\cH) = \omega_{\wt{0}}\cC(\cH) = \cW$. Recall that a $\cW$-compactification of $\Z$ is any pair $(\alpha,X)$ where $X$ is a compact right topological semigroup (necessarily a group in this case) and $\alpha \colon \Z \to X$ is a homomorphism such that $\alpha(\Z)$ is dense in $X$, $\alpha(\Z) \subseteq \Lambda(X)$, and $\alpha^{\ast}\cC(X) = \cW$.

Keeping in mind the representation theory of point transitive systems, the universal minimal system with quasi-discrete spectrum is obtained by combining Theorems \ref{sec:qdsweyl} and \ref{sec:hweyl}:

\begin{cor}
\label{sec:universal}
The system $(\cH,T)$ is the universal minimal system with quasi-discrete spectrum.
\end{cor}

\begin{rems}
In \cite{brown69}, it is shown that all minimal systems with quasi-discrete spectrum are factors of (what is essentially) $(\{\id_{\T}\} \times H(\T)^{\N},T)$. This is not a minimal system in itself, but $(\cH,T)$ is its factor (in addition to being a subsystem). We can construct a homomorphism $\pi \colon \{\id_{\T}\} \times H(\T)^{\N} \to \cH$ so that, for any $\phi \in \{\id_{\T}\} \times H(\T)^{\N}$, the element $\pi(\phi) \in \cH$ is uniquely determined by the properties $\pi(\phi)_1 = \phi_1$ and $\pi(\phi)_k(\exx{(b/n!)}) = \phi_k(\exx{(b/n!)})$ for all $k, n \in \N$ and $b \in B$ (the set $B$ is the basis for $\R$ over $\Q$ without $1$).

Note that we do not need the whole group $G(\cH,T)$ to separate the points of $\cH$. Equivalently, we do not need the whole group $\cP$ to generate $\cW$. For example, we could consider $\cP^{\prime} = \cP_1 \cup \cR$ where $\cR$ stands for the set of those functions $f \in \cP$ whose corresponding polynomial $\theta_0 + \sum_{k = 1}^m \theta_k n^k$ has at least one irrational coefficient $\theta_k$, $k \neq 0$. Then, $\lspan \cP^{\prime}$ is dense in $\cW$. This follows from the fact that the functions of $\cP \setminus \cR$ are periodic and, therefore, already included in $\mathcal{AP}(\Z) = \ov{\lspan \cP_1}$. This redundancy is reflected in the condition \eqref{eq:H.inf}.
\end{rems}

\begin{thm}
\label{sec:uew}
Any factor of $(\cH,T)$ is uniquely ergodic. In particular, any minimal system with quasi-discrete spectrum is uniquely ergodic.
\end{thm}

The proof of this theorem is a straightforward consequence of Theorems~\ref{sec:ueambit}, \ref{sec:poly} and \ref{sec:hweyl}. Special cases of it appear in \cite{hahnparry65} and \cite{salehi91}.


\section{Factors of $(\cH,T)$}

The study of the factors of $(\cH,T)$ reduces to the study of certain subgroups of $\cH$ since the latter can be identified with $E(\cH,T)$ via the mapping $\phi \mapsto \lambda_\phi : \cH \to E(\cH,T)$, which is an isomorphism of groups and of flows. More precisely, given a factor $(X,s)$ of $(\cH,T)$ and a homomorphism $\pi \colon \cH \to X$, we can define an action of $\cH$ on $X$ by $\phi x = \wt{\pi}(\lambda_\phi)x$ for all $\phi \in \cH$ and $x \in X$, and letting $x_0 = \pi(\wt{0})$, we obtain a subgroup $\cG$ of $\cH$ defined by
	\[\cG = \set{\phi \in \cH}{\phi x_0 = x_0}.
\]
This group coincides with the preimage $\pi^{-1}(x_0)$, and more generally, $\pi^{-1}(\phi x_0) = \phi \star \cG$ for any $\phi \in \cH$. The question whether $(X,s)$ has quasi-discrete spectrum or not can be answered in a somewhat obvious way in terms of the quasi-eigenfunctions of $(\cH,T)$ and the group $\cG$:

\begin{prop}
Let $(X,s)$, $\pi \colon \cH \to X$, $x_0$ and $\cG$ be as above. Then, $(X,s)$ has quasi-discrete spectrum if and only if there exists a subgroup $F$ of $G(\cH,T)$ with the property that, for any $\phi, \psi \in \cH$, the identity $f(\phi) = f(\psi)$ obtains for all $f \in F$ if and only if $\phi^{-1} \star \psi \in \cG$.
\end{prop}

\begin{proof}
If $(X,s)$ has quasi-discrete spectrum, then we can put $F = \pi^{\ast} G(X,s)$. The functions of $F$ are constant on the $\pi$-fibres, i.e., $f(\phi) = f(\psi)$ whenever $f \in F$ and $\phi, \psi \in \cH$ are such that $\pi(\phi) = \pi(\psi)$ or, equivalently, $\phi^{-1} \star \psi \in \cG$. If $\phi, \psi \in \cH$ are such that $\phi^{-1} \star \psi \notin \cG$, then $\pi(\phi) \neq \pi(\psi)$, and we can find $f \in F$ with $\pi^{\ast}f(\phi) \neq \pi^{\ast}f(\psi)$.

For the converse, suppose that there exists a subgroup $F$ of $G(\cH,T)$ with the property stated in the claim. Then, we can find a group $H \subseteq G(X,s)$ so that $\pi^{\ast}H = F$ (Proposition~\ref{sec:qeffactor}), and $H$ separates points. In particular, the quasi-eigenfunctions of $(X,s)$ separate points, and $(X,s)$ has quasi-discrete spectrum.
\end{proof}

Therefore, to find an example of a factor of $(\cH,T)$ that does not have quasi-discrete spectrum, we must find a subgroup $\cG$ of $\cH$ so that $\cH/\cG$ is Hausdorff and with the property that the cosets in $\cH/\cG$ cannot be separated by those quasi-eigenfunctions of $(\cH,T)$ which are constant on the cosets. (Recall that, if $\cG$ is a subgroup of $\cH$ for which $\cH/\cG$ is Hausdorff, the latter is the phase space of a factor of $(\cH,T)$.)

We start by fixing some $x \in \T \setminus \T_Q$. Define $\cG \subseteq \cG_1 \subseteq \cH$ as follows:
\begin{align*}
\cG_1 &= \set{\phi \in \cH}{\phi_1(x^2) = \phi_1(y) = 1 \text{ for all } y \in \T_Q}, \\
\cG &= \set{\phi \in \cG_1}{\phi_2(x) = 1}.
\end{align*}
The set $\cG$ contains at least the identity element $\wt{0}$, so both $\cG$ and $\cG_1$ are non-empty. They are closed under the operation $\star$. Indeed, if $\phi, \psi \in \cG_1$ and $y \in \T_Q$, then
\begin{align*}
(\phi \star \psi)_1(x^2) &= \phi_1(x^2) \psi_1(x^2) = 1, \\
(\phi \star \psi)_1(y) &= \phi_1(y) \psi_1(y) = 1.
\end{align*}
If $\phi, \psi \in \cG$, then $\psi_1(x)^2 = 1$, so $\psi_1(x) \in \T_2 = \{1,-1\}$ and $\phi_1(\psi_1(x)) = 1$, implying that
	\[(\phi \star \psi)_2(x) = \phi_2(x) \psi_2(x) \phi_1(\psi_1(x)) = 1,
\]
as claimed. It is easy to check that, if $\psi \in \cH$, then $(\psi^{-1})_1 = \ov{\psi_1}$ and $(\psi^{-1})_2 = \ov{\psi_2}(\psi_1 \circ \psi_1)$. Thus, given $\phi \in \cG_1$, $\psi \in \cG$ and any $y \in \T_Q$,
\begin{align*}
(\phi^{-1})_1(x^2) &= \ov{\phi_1(x^2)} = 1, \\
(\phi^{-1})_1(y) &= \ov{\phi_1(y)} = 1, \\
(\psi^{-1})_2(x) &= \ov{\psi_2(x)}\psi_1(\psi_1(x)) = 1.
\end{align*}
These identities show that $\phi^{-1} \in \cG_1$ and $\psi^{-1} \in \cG$. In conclusion, $\cG$ and $\cG_1$ are subgroups of $\cH$.

Clearly, both $\cG$ and $\cG_1$ are closed, but we intend to show that they have the stronger property of $\cH/\cG$ and $\cH/\cG_1$ being Hausdorff, i.e., that the quotient relations are closed. Let $R(\cG) \subseteq R(\cG_1) \subseteq \cH \times \cH$ denote these relations, so $(\phi, \psi) \in R(\cG)$ if and only if $\phi^{-1} \star \psi \in \cG$, and $(\phi,\psi) \in R(\cG_1)$ if and only if $\phi^{-1} \star \psi \in \cG_1$. Equivalently, if $\phi,\psi \in \cH$, then $(\phi,\psi) \in R(\cG_1)$ if and only if the conditions
\begin{align*}
\phi_1(x^2) &= \psi_1(x^2), \tag{A} \label{eq:A} \\
\phi_1(y) &= \psi_1(y) \text{ for all } y \in \T_Q \tag{B} \label{eq:B}
\end{align*}
are in effect, and $(\phi,\psi) \in R(\cG)$ if and only if $(\phi,\psi) \in R(\cG_1)$ and also
\begin{equation*}
\phi_2(x)\phi_1(\psi_1(x)) = \psi_2(x)\phi_1(\phi_1(x)) \tag{C} \label{eq:C}.
\end{equation*}
Condition~\eqref{eq:C} appears asymmetric at first glance, but \eqref{eq:A} implies that $\ov{\phi_1} \psi_1 (x) \in \T_2$, so applying \eqref{eq:B} we get $\ov{\phi_1} \psi_1(\ov{\phi_1} \psi_1(x)) = 1$. This can be used to show that the roles of $\phi$ and $\psi$ are interchangeable in \eqref{eq:C}. The relation $R(\cG_1)$ is clearly closed. To prove that $R(\cG)$ is closed, we introduce a function $\alpha \colon R(\cG_1) \to \T_2$, $\alpha(\phi,\psi) = \phi_1(\phi_1(x) \ov{\psi_1(x)})$ for all $(\phi,\psi) \in R(\cG_1)$. Note that \eqref{eq:A} implies that $\phi_1(x) \ov{\psi_1(x)} \in \T_2$ for all $(\phi,\psi) \in R(\cG_1)$, so the range of $\alpha$ is indeed in $\T_2$. The function $\alpha$ is continuous: if $(\phi_\lambda,\psi_\lambda) \conv{\lambda} (\phi,\psi)$ in $R(\cG_1)$, then we can find $\lambda_0$ so that $(\phi_\lambda)_1(x) \ov{(\psi_\lambda)_1(x)} = \phi_1(x) \ov{\psi_1(x)}$ for $\lambda \geq \lambda_0$, and we can also find $\lambda_1$ so that $(\phi_\lambda)_1(-1) = \phi_1(-1)$ for $\lambda \geq \lambda_1$, and therefore $\alpha(\phi_\lambda,\psi_\lambda) = \alpha(\phi,\psi)$ for all $\lambda \geq \lambda_0, \lambda_1$. Condition \eqref{eq:C} can be stated equivalently as
\begin{equation*}
\phi_2(x) = \psi_2(x) \alpha(\phi,\psi), \tag{C$^{\prime}$} \label{eq:Cprime}
\end{equation*}
$(\phi,\psi) \in R(\cG_1)$. It is now clear that $R(\cG)$ is also closed.

It remains to show that the quasi-eigenfunctions of $(\cH,T)$ that are constant on the cosets of $\cH/\cG$ do not separate these cosets. Consider any quasi-eigenfunction $q(v)$ of $(\cH,T)$, $v \in V$, for which $q(v)(\phi) = q(v)(\psi)$ for all $(\phi,\psi) \in R(\cG)$ or, equivalently, $q(v)(\phi) = q(v)(\phi \star \psi)$ for all $\phi \in \cH$ and $\psi \in \cG$. Suppose that $v_k = 1$ for all $k > n$ for some $n \in \N$. Now, for any $\phi \in \cH$ and $\psi \in \cG$,
\begin{align*}
\prod_{k = 0}^n \phi_k(v_k) &= q(v)(\phi) = q(v)(\phi \star \psi) = \prod_{k = 0}^n (\phi \star \psi)_k(v_k) = \prod_{k = 0}^n \prod_{j = 0}^k \phi_{k-j}(\psi_j(v_k)) \\
&= \left[ \prod_{k = 0}^n \phi_k(v_k) \right] \left[ \prod_{k = 1}^n \prod_{j = 1}^k \phi_{k-j}(\psi_j(v_k)) \right],
\end{align*}
and cancelling the common terms,
\begin{equation}
\prod_{k = 1}^n \prod_{j = 1}^k \phi_{k-j}(\psi_j(v_k)) = 1. \label{eq:reduction}
\end{equation}
The substitution $\phi = \wt{0}$ produces $\prod_{k = 1}^n \psi_k(v_k) = 1$ for all $\psi \in \cG$. As a consequence, $v_k \in \T_Q$ for all $k \geq 3$ since, for any $y \in \T \setminus \T_Q$, $\phi_k(y)$ can attain any prescribed value in $\T$ for a suitably chosen $\psi \in \cG$ when $k \geq 3$. Also, since $\psi_k(y) = 1$ for all $\psi \in \cG$, $k \in \N$ and $y \in \T_Q$ (as follows from the condition \eqref{eq:H.inf} in conjunction with the identity $\psi_1(y) = 1$, $\psi \in \cG$, $y \in \T_Q$), we get $\psi_1(v_1) \psi_2(v_2) = 1$ for all $\psi \in \cG$. Taking these observations into account, the substitution $\phi = \wt{1}$ in \eqref{eq:reduction} yields $\psi_1(v_2) = 1$ for all $\psi \in \cG$. We infer that $v_2 = x^{2q}z$ for some $q \in \Z$ and $z \in \T_Q$; otherwise, we could simply pick some $\xi \in H(\T)$ so that $\xi(v_2) \neq 1$ and then adapt the extension techniques from the proof of Lemma~\ref{sec:horbit} to find $\psi \in \cG$ with $\psi_1 = \xi$, resulting in a contradiction. Thus, $\psi_2(v_2) = 1$ and also $\psi_1(v_1) = 1$ for all $\psi \in \cG$. Again, $v_1 = x^{2p}w$ for some $p \in \Z$ and $w \in \T_Q$. For all $\psi \in \cG$, we have $q(v)(\psi) = v_0$. It is evident that any $v \in V$ satisfying the discussed conditions for $v_1$ and $v_2$ also satisfies \eqref{eq:reduction} for any $\phi \in \cH$ and $\psi \in \cG$, so the resulting quasi-eigenfunction $q(v)$ is constant on the cosets of $\cH/\cG$.

Pick any $\phi \in \cH$ such that $\phi_1 = 0^{\times}$ and $\phi_2(x) = -1$. Note that $\phi_k(y) = 1$ for all $k \in \N$ and $y \in \T_Q$ by \eqref{eq:H.inf}. Now, $\phi \notin \cG$, so the cosets $\wt{0} \star \cG$ and $\phi \star \cG$ are distinct. However, if $v \in V$ is as above,
	\[q(v)(\phi) = v_0 \phi_1(v_1) \phi_2(v_2) = v_0 (-1)^{2q} = v_0 = q(v)(\psi)
\]
for any $\psi \in \cG$. In other words, the cosets $\wt{0} \star \cG$ and $\phi \star \cG$ cannot be separated by quasi-eigenfunctions that are constant on all cosets of $\cH/\cG$. The construction is thus complete.

\begin{rems}
The fact that the group $\cG$ produces a Hausdorff quotient $\cH/\cG$ follows also from the fact that $\cG$ is $\tau$-closed in $\cH$, and showing that $\cG$ has this property is marginally easier than the direct proof presented above. Recall that the $\tau$-closedness of a set $K \subseteq \cH$ means that, if $(n_\lambda)$ is a net in $\Z$ so that $\epsilon(n_\lambda) \conv{\lambda} u$ for some fixed minimal idempotent $u \in \beta \Z$ (the mapping $\epsilon \colon \Z \to \beta \Z$ being the usual homomorphism of the integers to the Stone-\v{C}ech compactification) and if $(\kappa_\lambda)$ is a net in $K$ so that $T^{n_\lambda} \kappa_\lambda \conv{\lambda} \kappa$ in $\cH$, then $\kappa \in K$. In general, given a minimal distal system $(X,s)$ (or a minimal distal flow with a topological group action) and some $\tau$-closed subgroup $\cG$ of $E(X,s)$, the quotient $E(X,s)/\cG$ is Hausdorff and, therefore, the phase space of a factor of $(E(X,s),\lambda_s)$. The reader may consult \cite{auslander} or \cite{devries} for more information on the $\tau$-topology, but we shall avoid the explicit use of this concept.

Another point to be made about this example is that the space $\cH/\cG$ is not connected. The condition \eqref{eq:B} is the cause of this. Therefore, the factor thus obtained is not totally minimal: if $s\colon \cH/\cG \to \cH/\cG$ is the homeomorphism for which $(\cH/\cG,s)$ is a factor of $(\cH,T)$ via the quotient map, then the system $(\cH/\cG,s^n)$ is not minimal for any $n \geq 2$. The corresponding concept in the measure-theoretic setting is total ergodicity, and as mentioned in the introduction, the factors of a totally ergodic measure preserving transformation on a standard probability space with (measure-theoretic) quasi-discrete spectrum inherit this property. It is not clear whether the example provided here has quasi-discrete spectrum in the measure-theoretic sense with respect to its unique invariant measure, and the question goes beyond the purely topological programme of this paper.
\end{rems}

We now turn to the problem of finding the Ellis groups of factors of $(\cH,T)$. So suppose that $(X,s)$ is some factor of $(\cH,T)$, and let $\pi \colon \cH \to X$ be a homomorphism. Let $x_0 = \pi(\wt{0})$, and let $\cG$ be defined with respect to $x_0$ as in the beginning of this section. If we identify $\cH$ with $E(\cH,T)$ in the canonical way, the induced homomorphism corresponding to $\pi$ can be defined as a mapping $\wt{\pi} \colon \cH \to E(X,s)$, $\wt{\pi}(\phi)(x) = \phi x$, $\phi \in \cH$, $x \in X$. Let $\cK$ be the kernel of $\wt{\pi}$, that is,
	\[\cK = \set{\phi \in \cH}{\phi x = x \text{ for all } x \in X}.
\]
Another way of describing $\cK$ is provided by the identification of $X$ with $\cH/\cG$:
	\[\cK = \set{\phi \in \cH}{\phi \star \psi \star \cG = \psi \star \cG \text{ for all } \psi \in \cH}.
\]
The Ellis group $E(X,s)$ is topologically isomorphic with the quotient $\cH/\cK$. Obviously, $\cK$ is normal and contained in $\cG$. We intend to show that $(E(X,s),\lambda_s)$ has quasi-discrete spectrum or, equivalently, that the quasi-eigenfunctions of $\cH$ that are constant on the cosets of $\cH/\cK$ separate these cosets. What is required is a sufficiently detailed description of $\cK$.

Recall that, for any $m \in \N$, we have defined a subgroup $\cH_{m,1}$ of $\cH_m$,
	\[\cH_{m,1} = \set{\phi \in \cH_m}{\phi_k = \wt{0}_k \text{ for all } 0 \leq k \leq m-1}.
\]
This is closed, abelian and normal in $\cH_m$. 

\begin{lem}
\label{sec:normalextend}
Let $\cN$ be a closed, normal subgroup of $\cH$. If $\rho_m(\cN)$ contains $\cH_{m,1}$ for some $m \in \N$, then $\rho_n(\cN)$ contains $\cH_{n,1}$ for every $n \geq m$.
\end{lem}

\begin{proof}
It suffices to prove that the inclusion $\cH_{m,1} \subseteq \rho_m(\cN)$, $m \in \N$, implies $\cH_{m+1,1} \subseteq \rho_{m+1}(\cN)$. So suppose that the former holds. By normality, given any $\phi \in \cN$ and $\psi \in \cH$, the commutator $[\phi,\psi] = \phi^{-1} \star \psi^{-1} \star \phi \star \psi$ is in $\cN$, so $\rho_{m+1}([\phi,\psi]) \in \rho_{m+1}(\cN)$. If we choose $\phi \in \cN$ so that $\rho_m(\phi) \in \cH_{m,1}$, we can use Proposition~\ref{sec:commutator} with $k = m-1$ to see that $[\phi,\psi]_j = 0^{\times}$ for all $1 \leq j \leq m$ and $[\phi,\psi]_{m+1} = \ov{(\psi_1 \circ \phi_m)}(\phi_m \circ \psi_1)$. In this case, $\rho_{m+1}([\phi,\psi]) \in \cH_{m+1,1}$ for all $\psi \in \cH$.

Let $\cN_0 = \cH_{m+1,1} \cap \rho_{m+1}(\cN)$, so $\cN_0$ is the closed, abelian subgroup of $\rho_{m+1}(\cN)$ consisting of those elements $\xi$ for which $\xi_k = \wt{0}_k$ for all $0 \leq k \leq m$. Let $N_0 \subseteq H(\T)$ be the closed subgroup of those elements $\alpha$ for which there exists some $\xi \in \cN_0$ with $\xi_{m+1} = \alpha$. The arguments above show that $[\phi,\psi]_{m+1} \in N_0$ for those $\phi \in \cN$ for which $\rho_m(\phi) \in \cH_{m,1}$ and for all $\psi \in \cH$. Character theory tells us that $N_0$ can be written as the intersection of the kernels of a group of characters of $H(\T)$, so there is a subgroup $\Gamma$ of $\T$ so that
	\[N_0 = \set{\alpha \in H(\T)}{\alpha(x) = 1 \text{ for all } x \in \Gamma}.
\]
Keeping in mind the condition \eqref{eq:H.inf}, the proof is complete once we show that $\Gamma = \T_Q$, for then $\rho_{m+1}(\cN)$ contains all members of $\cH_{m+1,1}$. Firstly, \eqref{eq:H.inf} implies that $\T_Q \subseteq \Gamma$ since, for all $\alpha \in N_0$ and $x \in \T_Q$, we must have $\alpha(x) = 1$. Secondly, for any $\alpha, \beta \in H(\T)$ with $\alpha(x) = 1$ for all $x \in \T_Q$, we can find $\phi \in \cN$ and $\psi \in \cH$ so that $\rho_m(\phi) \in \cH_{m,1}$, $\phi_m = \alpha$ (by the assumption that $\cH_{m,1} \subseteq \rho_m(\cN)$) and $\psi_1 = \beta$, and then, for all $y \in \Gamma$,
	\[\ov{\beta(\alpha(y))} \alpha(\beta(y)) = [\phi,\psi]_{m+1}(y) = 1,
\]
i.e., $\alpha(\beta(y)) = \beta(\alpha(y))$. This condition implies that $\Gamma \subseteq \T_Q$, and we show this by an argument by contradiction. Suppose that there is some $y \in \Gamma \setminus \T_Q$. Then, we can choose $\alpha \in H(\T)$ so that $\alpha(y) = y$ (and $\alpha(x) = 1$ for all $x \in \T_Q$), and we can choose $\beta \in H(\T)$ so that $\beta(y) = -1$. This leads to a contradiction: $\alpha(\beta(y)) = \alpha(-1) = 1$, but $\beta(\alpha(y)) = \beta(y) = -1$. In conclusion, $\Gamma = \T_Q$, as desired.
\end{proof}

The next lemma can be thought of as a `$\tau$-free' alternative to Corollary 5 in chapter 14 of \cite{auslander}.

\begin{lem}
\label{sec:RAB}
Let $X$ be a group and a compact space, and let $A, B \subseteq X$ be subgroups so that the quotient spaces $X/A$ and $X/B$ are Hausdorff. Suppose that $B$ is normal, so $AB$ is a subgroup of $X$. Then, the quotient space $X/(AB)$ is Hausdorff.
\end{lem}

\begin{proof}
Let $R(G)$ be the quotient relation on $X$ for $G \in \{A,B,AB\}$, that is,
	\[R(G) = \set{(x,y) \in X \times X}{x^{-1}y \in G}.
\]
Assuming that $R(A)$ and $R(B)$ are closed, we must show that $R(AB)$ is closed. Pick $(x,y) \in \ov{R(AB)}$, and let $(x_\lambda)$ and $(y_\lambda)$ be nets in $X$ so that $(x_\lambda,y_\lambda) \in R(AB)$ for every $\lambda$ and $(x_\lambda,y_\lambda) \conv{\lambda} (x,y)$. We can find nets $(a_\lambda)$ and $(b_\lambda)$ in $A$ and $B$, respectively, so that $x_\lambda^{-1} y_\lambda = a_\lambda b_\lambda^{-1}$ for all $\lambda$. Now, $x_\lambda a_\lambda = y_\lambda b_\lambda$, $(x_\lambda,x_\lambda a_\lambda) \in R(A)$ and $(y_\lambda,y_\lambda b_\lambda) \in R(B)$ for every $\lambda$. By passing to a convergent subnet if necessary, we can assume that $x_\lambda a_\lambda \conv{\lambda} z$, and then also $y_\lambda b_\lambda \conv{\lambda} z$. Since $R(A)$ and $R(B)$ are closed, we can find $a \in A$ and $b \in B$ such that $xa = z = yb$, so $x^{-1}y = ab^{-1} \in AB$, showing that $(x,y) \in R(AB)$. This proves that $R(AB)$ is closed.
\end{proof}

The group $\cK$ discussed before the lemmas can now be described in a useful way.

\begin{thm}
\label{sec:normalstructure}
Let $\cG \subseteq \cH$ be a subgroup for which the quotient space $\cH/\cG$ is Hausdorff, and let $\cK \subseteq \cH$ be the normal subgroup
	\[\cK = \set{\phi \in \cH}{\phi \star \psi \star \cG = \psi \star \cG \text{ for all } \psi \in \cH}.
\]
\begin{itemize}
\item[(i)] If there exists a smallest number $m \in \N$ for which $\rho_m^{-1}(\rho_m(\cG)) = \cG$, then
	\[\cK = \set{\phi \in \cG}{\phi_k = \wt{0}_k \text{ for all } 0 \leq k \leq m-1}.
\]
\item[(ii)] If $\cG$ is a proper subset of $\rho_m^{-1}(\rho_m(\cG))$ for every $m \in \N$, then $\cK = \{\wt{0}\}$.
\end{itemize}
\end{thm}

\begin{proof}
We start with the first case, so suppose that there exists a smallest number $m \in \N$ for which $\rho_m^{-1}(\rho_m(\cG)) = \cG$. Then, we also have $\rho_m^{-1}(\rho_m(\psi \star \cG)) = \psi \star \cG$ for every $\psi \in \cH$ since $\rho_m$ is a group homomorphism. Let $\cL$ be the set
	\[\cL = \set{\phi \in \cG}{\phi_k = \wt{0}_k \text{ for all } 0 \leq k \leq m-1}.
\]
Pick $\phi \in \cL$. For an arbitrary $\psi \in \cH$, we have $(\phi \star \psi)_k = \psi_k = (\psi \star \phi)_k$ for $0 \leq k \leq m-1$ and $(\phi \star \psi)_m = \phi_m \psi_m = (\psi \star \phi)_m$. Hence, $\rho_m(\phi \star \psi) = \rho_m(\psi \star \phi)$, and 
	\[\rho_m(\phi \star \psi \star \cG) = \rho_m(\psi \star \phi \star \cG) = \rho_m(\psi \star \cG).
\]
It follows that $\phi \star \psi \star \cG = \psi \star \cG$, so $\phi \in \cK$.

To prove the converse, we argue by contradiction and assume that there is some $\phi \in \cK \setminus \cL$. Then, $\phi_k \neq \wt{0}_k$ for some $0 \leq k \leq m-1$. Obviously, $\phi_0 = 1^{\times} = \wt{0}_0$, so $m \geq 2$, and there is a smallest number $1 \leq n \leq m-1$ for which $\phi_n \neq 0^{\times}$. Observe that, for all $j \in \Z$, the $j$th iterate of $\phi$ satisfies $(\phi^j)_k = \wt{0}_k$ for $0 \leq k \leq n-1$ and $(\phi^j)_n = (\phi_n)^j$. We are free to require that $\phi_n$ is not continuous. Indeed, if it happens that $\phi_n = l^{\times}$ for some $l \in \Z \setminus \{0\}$, then the sequence $((\phi^j)_n)_j = ((jl)^{\times})_j$ clusters at some non-continuous element of $H(\T)$. Thus, we can find a cluster point $\phi^{\prime} \in \cK$ of $(\phi^j)$ so that $\phi^{\prime}_n$ is not continuous. Obviously, $\phi^{\prime}_k = \wt{0}_k$ for $0 \leq k \leq n-1$, and $\phi^{\prime}_n \neq 0^{\times}$, so we can replace $\phi$ by $\phi^{\prime}$ if necessary.

We intend to show that $\rho_{n+1}(\cK)$ contains $\cH_{n+1,1}$. Adapting the ideas of the proof of Lemma~\ref{sec:normalextend}, define $\cK_0 = \cH_{n+1,1} \cap \rho_{n+1}(\cK)$, a closed, abelian subgroup of $\cH_{n+1}$, and let $K_0$ be the group of all $\alpha \in H(\T)$ for which there is some $\xi \in \cK_0$ satisfying $\xi_{n+1} = \alpha$. Note that, for an arbitrary $\psi \in \cH$, the commutator $[\phi,\psi] = \phi^{-1} \star \psi^{-1} \star \phi \star \psi$ must be in $\cK$ by normality. From Proposition~\ref{sec:commutator}, we get
\begin{align*}
(\phi^{-1} \star \psi^{-1} \star \phi \star \psi)_j &= 0^{\times} \text{ for all } 1 \leq j \leq n, \\
(\phi^{-1} \star \psi^{-1} \star \phi \star \psi)_{n+1} &= \ov{(\psi_1 \circ \phi_n)}(\phi_n \circ \psi_1).
\end{align*}
Consequently, the group $K_0$ contains $\ov{(\psi \circ \phi_n)}(\phi_n \circ \psi)$ for any $\psi \in H(\T)$. Let $\Gamma \subseteq \T$ be the subgroup for which
	\[K_0 = \set{\alpha \in H(\T)}{\alpha(x) = 1 \text{ for all } x \in \Gamma}.
\]
Now,
\begin{equation}
\phi_n(\psi(x)) = \psi(\phi_n(x)) \text{ for all } \psi \in H(\T), x \in \Gamma. \label{eq:commutationx}
\end{equation}
By \eqref{eq:H.inf} and the construction of $\cK_0$, we have $\T_Q \subseteq \Gamma$. We must show that $\Gamma = \T_Q$, and we argue by contradiction. Suppose that there exists some $x \in \Gamma \setminus \T_Q$. Since $\phi_n$ is not continuous, we can find a net $(y_\lambda)$ in $\T$ with a limit $y$ so that $\phi_n(y_\lambda) \conv{\lambda} z$ for some $z \neq \phi_n(y)$. There is also a net $(\psi_\lambda)$ in $H(\T)$ so that $\psi_\lambda(x) = y_\lambda$ for every $\lambda$. By passing to a subnet if necessary, we can assume that $\psi_\lambda \conv{\lambda} \psi$ for some $\psi \in H(\T)$ with $\psi(x) = y$. Using \eqref{eq:commutationx}, we obtain
	\[\phi_n(y_\lambda) = \phi_n(\psi_\lambda(x)) = \psi_\lambda(\phi_n(x)) \conv{\lambda} \psi(\phi_n(x)) = \phi_n(\psi(x)) = \phi_n(y),
\]
a contradiction. Hence, $\Gamma = \T_Q$, as desired.

According to Lemma~\ref{sec:normalextend}, the group $\rho_k(\cK)$ contains $\cH_{k,1}$ for all $k \geq n+1$. This entails that $\cH_{k,1} \subseteq \rho_k(\cG)$ for all $k \geq n+1$. As a further corollary, we get $\rho_n^{-1}(\rho_n(\cG)) = \cG$. To see the inclusion of the former in the latter, pick an arbitrary $\psi \in \rho_n^{-1}(\rho_n(\cG))$. We can construct a sequence $(\psi_{(j)})_{j = 0}^{\infty}$ in $\cG$ in a recursive manner so that $\rho_{n+j}(\psi_{(j)}) = \rho_{n+j}(\psi)$ for every $j \in \Z_+$. Pick $\psi_{(0)} \in \cG$ so that $\rho_n(\psi_{(0)}) = \rho_n(\psi)$. Suppose that $\psi_{(j)} \in \cG$ has been defined and that $\rho_{n+j}(\psi_{(j)}) = \rho_{n+j}(\psi)$. Then, $(\psi \star \psi_{(j)}^{-1})_k = \wt{0}_k$ for all $0 \leq k \leq n+j$. Therefore, $\rho_{n+j+1}(\psi \star \psi_{(j)}^{-1}) \in \cH_{n+j+1,1} \subseteq \rho_{n+j+1}(\cG)$, and we can find some $\xi \in \cG$ so that $\rho_{n+j+1}(\psi \star \psi_{(j)}^{-1}) = \rho_{n+j+1}(\xi)$. Put $\psi_{(j+1)} = \xi \star \psi_{(j)} \in \cG$. It has the property $\rho_{n+j+1}(\psi_{(j+1)}) = \rho_{n+j+1}(\psi)$. The sequence $(\psi_{(j)})$ converges to $\psi$, so the latter is in $\cG$. Hence, $\rho_n^{-1}(\rho_n(\cG)) = \cG$. However, this contradicts the assumptions that $m \in \N$ is the smallest number satisfying $\rho_m^{-1}(\rho_m(\cG)) = \cG$ and $n \leq m-1$. This completes the proof of (i).

We can use (i) to prove (ii). Suppose that $\cG$ is a proper subgroup of $\cG_m = \rho_m^{-1}(\rho_m(\cG))$ for every $m \in \N$. It is easy to show that $\cG_m = \cG \star \ker \rho_m$ for any $m \in \N$. The group $\ker \rho_m$, again for any $m \in \N$, is normal in $\cH$, and the quotient group $\cH/\ker \rho_m$ is topologically isomorphic with $\cH_m$, so it is Hausdorff. Lemma~\ref{sec:RAB} guarantees that $\cH/\cG_m$ is Hausdorff for all $m \in \N$. Let $\cK_m \subseteq \cH$ be defined by
	\[\cK_m = \set{\phi \in \cH}{\phi \star \psi \star \cG_m = \psi \star \cG_m \text{ for all } \psi \in \cH},
\]
$m \in \N$. Observe that $\cK = \bigcap_{m = 1}^{\infty} \cK_m$. Indeed, if $\phi \in \cK$, then for all $m \in \N$ and $\psi \in \cH$, we have
	\[\phi \star \psi \star \cG_m = \phi \star \psi \star \cG \star \ker \rho_m = \psi \star \cG \star \ker \rho_m = \psi \star \cG_m,
\]
so $\phi \in \bigcap_{m = 1}^{\infty} \cK_m$. Conversely, if $\phi \in \bigcap_{m = 1}^{\infty} \cK_m$, then for all $\psi \in \cH$,
	\[\phi \star \psi \star \cG = \bigcap_{m \in \N} \phi \star \psi \star \cG \star \ker \rho_m = \bigcap_{m \in \N} \psi \star \cG \star \ker \rho_m = \psi \star \cG,
\]
showing that $\phi \in \cK$. By (i), we have $\phi_k = \wt{0}_k$ for any $m \in \N$, $\phi \in \cK_m$ and $0 \leq k \leq m-1$, so $\cK = \{\wt{0}\}$, and the proof is complete.
\end{proof}

\begin{rems}
\label{sec:wellis}
Case (i) of this theorem has an easily proved converse: if $m \in \N$ and $K_0 \subseteq H(\T)$ is a closed subgroup, then the set $\cK \subseteq \cH$ defined by
	\[\cK = \set{\phi \in \cH}{\phi_k = \wt{0}_k \text{ for all } 0 \leq k \leq m-1, \phi_m \in K_0}
\]
is a normal subgroup of $\cH$, and $\cH/\cK$ is a Hausdorff space. Thus, we have found all normal subgroups of $\cH$ that induce a Hausdorff quotient, and this means that we have also found all Ellis groups of factors of $(\cH,T)$. Such an Ellis group is of one of three types: it is either compact, abelian and monothetic (equicontinuous case), the full group $\cH$, or analogous to the Ellis group of a minimal $(m,x_0)$-system for $m \geq 2$, i.e., it is essentially $\cH_{m-1} \times G$ for some compact, abelian, monothetic group $G$ (possibly trivial), and the group operation is also similar to $\ast$ as defined in case (iii) of Theorem~\ref{sec:findim}. 

Case (ii) explains why the Ellis group of a minimal $(\infty,x_0)$-system $(\T^{\N},s)$ is isomorphic to $(\cH,T)$: we can define a homomorphism $\pi \colon \cH \to \T^{\N}$ by $\pi(\phi)_k = \phi_k(x_0)$, $k \in \N$, for all $\phi \in \cH$, and then
	\[\cG = \pi^{-1}(1,1,1,\ldots) = \set{\phi \in \cH}{\phi_k(x_0) = 1 \text{ for all } k \in \N}.
\]

\end{rems}

The main theorem of this section follows quite effortlessly from the one above:

\begin{thm}
\label{sec:wsystem}
Let $(X,s)$ be a factor of $(\cH,T)$. Then, the system $(E(X,s),\lambda_s)$ has quasi-discrete spectrum.
\end{thm}

\begin{proof}
Let $\pi \colon \cH \to X$ be a homomorphism, let $\cG = \pi^{-1}(\pi(\wt{0})) \subseteq \cH$, and let $\cK \subseteq \cH$ be the normal subgroup
	\[\cK = \set{\phi \in \cH}{\phi \star \psi \star \cG = \psi \star \cG \text{ for all } \psi \in \cH},
\]
so $E(X,s)$ is topologically isomorphic to $\cH/\cK$. We must find a group $Q$ of quasi-eigenfunctions of $(\cH,T)$ that are constant on the cosets of $\cH/\cK$ and separate these cosets. By Theorem~\ref{sec:normalstructure}, there are two cases to consider: either there is a smallest number $m \in \N$ for which $\rho_m^{-1}(\rho_m(\cG)) = \cG$, in which case
	\[\cK = \set{\phi \in \cG}{\phi_k = \wt{0}_k \text{ for all } 0 \leq k \leq m-1},
\]
or no such number $m \in \N$ exists and $\cK = \{\wt{0}\}$. The latter case is easy; simply choose $Q = G(\cH,T)$. To handle the remaining case, suppose that $\rho_m^{-1}(\rho_m(\cG)) = \cG$ for $m \in \N$, the smallest number with this property, and that $\cK$ is as stated above. Let $K_0 \subseteq H(\T)$ be the closed group of those elements $\alpha \in H(\T)$ for which there exists some $\phi \in \cK$ with $\phi_m = \alpha$. Let $\Gamma \subseteq \T$ be the subgroup for which
	\[K_0 = \set{\alpha \in H(\T)}{\alpha(x) = 1 \text{ for all } x \in \Gamma}
\]
Recall that $V$ is used to denote the weak product of $\Z_+$ copies of $\T$ and that we have defined a group isomorphism $q \colon V \to G(\cH,T)$. Let $W \subseteq V$ be the subgroup
	\[W = \set{w \in V}{w_m \in \Gamma, w_k = 1 \text{ for all } k \geq m+1}.
\]
It is easy to see that, for any $w \in W$, $\phi \in \cH$ and $\psi \in \cK$, we have the identity $q(w)(\phi) = q(w)(\phi \star \psi)$. In other words, the functions in $Q = q(W)$ are constant on the cosets of $\cH/\cK$. To show that distinct cosets can be separated by the members of $Q$, suppose that $\phi, \psi \in \cH$ are such that $q(w)(\phi) = q(w)(\psi)$ for all $w \in W$. Then, going through all $w \in W$ such that $w_m = 1$, we get $\phi_k = \psi_k$ for all $0 \leq k \leq m-1$. It follows that $(\phi^{-1} \star \psi)_k = \wt{0}_k$ for all $0 \leq k \leq m-1$. In addition, for any $x \in \Gamma$, choosing $w \in W$ so that $w_m = x$ and $w_k = 1$ for $k \neq m$, we get $(\phi^{-1} \star \psi)_m(x) = \ov{\phi_m} \psi_m(x) = 1$. Thus, $\rho_m(\phi^{-1} \star \psi) \in \rho_m(\cK) \subseteq \rho_m(\cG)$, and therefore, $\phi^{-1} \star \psi \in \rho_m^{-1}(\rho_m(\cG)) = \cG$. We conclude that $\phi^{-1} \star \psi \in \cK$, as desired, and the argument is complete.
\end{proof}

\begin{rem}
The example we constructed in the beginning of this section shows that the converse of Proposition~\ref{sec:qdslift} fails. In other words, there is a minimal distal system $(X,s)$ such that $(E(X,s),\lambda_s)$ has quasi-discrete spectrum but $(X,s)$ does not. 
\end{rem}

Theorem~\ref{sec:wsystem} has an interesting interpretation in the terminology of $m$-admissible subalgebras of $\cW$:

\begin{cor}
If $A \subseteq \cW$ is an $m$-admissible subalgebra, then $A = \ov{\lspan(A \cap \cP)}$.
\end{cor}

\begin{proof}
Since $A$ is $m$-admissible, there is a corresponding semigroup compactification $(\alpha,X)$ of $\Z$, that is, $X$ is a compact right topological semigroup, and $\alpha \colon \Z \to X$ is a semigroup homomorphism such that the image $\alpha(\Z)$ is dense in $X$, $\alpha(\Z) \subseteq \Lambda(X)$, and $\alpha^{\ast} \cC(X) = A$ (see \cite{bjm}). There is also a continuous, surjective semigroup homomorphism $\pi \colon \cH \to X$ with the property that $\pi \circ \tau = \alpha$ where $\tau \colon \Z \to \cH$ is the mapping $\tau(n) = \wt{n}$, $n \in \Z$. Consequently, $X$ is a group. Let $s = \lambda_{\alpha(1)}$. Then, $(X,s)$ is a minimal distal system, it is a factor of $(\cH,T)$, and $(E(X,s),\lambda_s)$ is isomorphic to $(X,s)$. By Theorem~\ref{sec:wsystem}, $(X,s)$ has quasi-discrete spectrum. Since $\omega_{\alpha(0)}\cC(X) = A$, Theorem~\ref{sec:qdsweyl} implies that $\lspan(A \cap \cP)$ is dense in $A$.
\end{proof}

Theorem~\ref{sec:wsystem} suggests a property of dynamical systems, which we call the $\cW$\emph{-property}:

\begin{defn}
A dynamical system $(X,s)$ is a $\cW$\emph{-system} if $(E(X,s),\lambda_s)$ has quasi-discrete spectrum. 
\end{defn}

Note that if $(E(X,s),\lambda_s)$ has quasi-discrete spectrum, it is automatically minimal and distal. A $\cW$-system is therefore distal. The Ellis groups of $\cW$-systems are given by Theorem~\ref{sec:normalstructure}. The following proposition should justify this choice of terms.

\begin{prop}
Let $(X,s)$ be a minimal system. The following conditions are equivalent:
\begin{itemize}
\item[(i)] The system $(X,s)$ is a $\cW$-system.
\item[(ii)] The system $(X,s)$ is a factor of $(\cH,T)$.
\item[(iii)] There exists a point $x \in X$ so that $\omega_x\cC(X) \subseteq \cW$.
\item[(iv)] For every point $x \in X$, we have $\omega_x\cC(X) \subseteq \cW$.
\end{itemize}
If $(X,s)$ is a dynamical system, not necessarily minimal, then (i), (iv) and the following condition are equivalent:
\begin{itemize}
\item[(v)] The system $(X,s)$ is distal, and its minimal subsystems are $\cW$-systems.
\end{itemize}
\end{prop}

The arguments of the proof in the minimal case are direct applications of Theorem~\ref{sec:hweyl}, Corollary~\ref{sec:universal}, Theorem~\ref{sec:wsystem} and the duality of point transitive dynamical systems and certain subalgebras of $l^{\infty}$. Obviously, $(\cH,T)$ is the universal minimal $\cW$-system. The non-minimal case follows from the well-known fact that the phase space of a distal system can be partitioned into minimal sets. The next proposition follows immediately from Theorem~\ref{sec:uew} and the proposition above.

\begin{prop}
A minimal $\cW$-system is uniquely ergodic.
\end{prop}

The following inheritance properties are easy to prove:

\begin{prop}
The $\cW$-property is inherited by factors and subsystems. Products and inverse limits of $\cW$-systems are $\cW$-systems.
\end{prop}

\bibliographystyle{amsplain}
\bibliography{biblio}

\end{document}